\documentclass[a4paper]{amsart}
\usepackage{amsmath,amsthm,amssymb,latexsym,epic,bbm,comment,color}
\usepackage{graphicx,enumerate,stmaryrd}
\usepackage[all,2cell]{xy}
\xyoption{2cell}

\usepackage{tikz-cd}

\newtheorem{theorem}{Theorem}
\newtheorem{lemma}[theorem]{Lemma}
\newtheorem{definition}[theorem]{Definition}
\newtheorem{corollary}[theorem]{Corollary}
\newtheorem{proposition}[theorem]{Proposition}

\newtheorem{example}[theorem]{Example}

\usepackage[all]{xy}
\usepackage[active]{srcltx}
\usepackage[parfill]{parskip}

\begin{document}
\title[Harish-Chandra modules over invariant subalgebras]
{Harish-Chandra modules over invariant subalgebras in a skew-group ring}
\author{Volodymyr Mazorchuk and Elizaveta Vishnyakova}

\begin{abstract}
We construct a new class of algebras resembling enveloping algebras and generalizing 
orthogonal Gelfand-Zeitlin algebras and rational Galois algebras studied by \cite{EMV,FuZ,RZ,Har}. 
The algebras are defined via a geometric realization in terms of sheaves of functions invariant
under an action of a finite group. A natural class of modules over these algebra can be 
constructed via a similar geometric realization. In the special case of a local reflection group, 
these modules are shown to have an explicit basis, generalizing similar results for orthogonal 
Gelfand-Zeitlin algebras from \cite{EMV} and for rational Galois algebras from \cite{FuZ}. 
We also construct a family of canonical simple Harish-Chandra modules and give sufficient conditions for simplicity of some modules. 
\end{abstract}

\maketitle

\section{Introduction}

In the last decade there was a significant progress in understanding infinite dimensional 
simple modules over the  Lie algebra $\mathfrak{gl}_n$, see e.g. \cite{FGR,Ni1,Ni2,EMV} and references therein.
An essential part of this progress is related to the study of so-called {\em Gelfand-Zeitlin modules}
which originate from \cite{DOF} based on \cite{GZ} (see \cite{EMV} for a detailed literature overview on
Gelfand-Zeitlin modules). Various approaches to the study of Gelfand-Zeitlin modules rely on different
realizations of the universal enveloping algebras which led to a number of generalizations of 
such algebras. These include {\em orthogonal Gelfand-Zeitlin algebras}
introduced in \cite{Ma} and {\em Galois algebras} introduced in \cite{FO}. These generalizations
include also finite W-algebras of type $A$, see \cite{Arak,Har}, and were studied in, in particular,  
\cite{EMV,Har,FuZ,RZ}.  The recent preprint \cite{KTWWY} establishes a
relation between orthogonal Gelfand-Zeitlin algebras and Khovanov-Lauda-Rouquier algebras from \cite{KL,Ro} 
and, in particular, leads to a (not very explicit) classification of simple 
Gelfand-Zeitlin modules over orthogonal Gelfand-Zeitlin algebras.

In the present paper we define and study a simultaneous geometric generalization of 
orthogonal Gelfand-Zeitlin algebras and Galois algebras. Both our construction and methods of study are 
inspired by the geometric approach of \cite{Vi,Vi2} to singular Gelfand-Zeitlin modules
and is formulated in elementary sheaf-theoretic terms. To any semidirect product $G \ltimes V$ of 
a finite group $G$ and a complex-analytic or linear algebraic group $V$, we associate the corresponding 
skew-group ring $\mathcal S$. We denote by $\mathcal O$ the sheaf of holomorphic or polynomial 
functions on $V$. There is a natural action of $G$ on $\mathcal O$ and it is natural to 
consider the sheaf $\mathcal O^G$ of $G$-invariants in $\mathcal O$. Main protagonists of the 
present paper are subalgebras  in $\mathcal S$ that preserve the sheaf $\mathcal O^G$.
Orthogonal Gelfand-Zeitlin algebras, Galois algebras, finite W-algebras and Galois orders,
studied in \cite{EMV,FuZ,RZ,Har}, are all special cases of our construction.
In a special case which we call {\em standard algebras of type $A$}, we give an explicit 
description of our algebras as subalgebras in the {\em universal ring} as introduced in \cite{Vi2}.
Our geometric approach also naturally provides a construction of a large family of (simple) modules
over our algebras, generalizing \cite{Vi,Vi2,EMV}. We note that, the general case of our 
construction seems to be outside the scope of {\em Harish-Chandra subalgebras and 
Gelfand-Zeitlin modules} as defined in \cite{DFO}. However, it still fits into the general
Harish-Chandra setup which studies modules over some algebra on which a certain subalgebra 
acts locally finite. In particular, our results significantly generalize and simplify 
many results from \cite{FuZ}.

The paper is organized as follows: Section~\ref{s2} contains a description of our
setup and  preliminaries. Section~\ref{s3} defines and provides basic structure results for 
our algebras. Sections~\ref{s4}, \ref{s5} and \ref{s7} study in detail the spacial case of rational Galois orders.
Sections~\ref{s6} describes applications of our approach to the study of Gelfand-Zeitlin modules.
Finally, in Sections~\ref{s8} we construct canonical simple Harish-Chandra modules over our algebras and
give a sufficient condition for simplicity of these modules.
\bigskip

\textbf{Acknowledgements:} V.~M. is partially supported by
the Swedish Research Council and G{\"o}ran Gustafsson Stiftelse.
E.~V. is partially  supported by SFB TR 191, Tomsk State University, 
Competitiveness Improvement Program and by Programa Institucional de
Aux\'{\i}lio \`{a} Pesquisa de Docentes Rec\'{e}m-Contratados ou Rec\'{e}m-Doutorados,
UFMG 2018. We thank D.~Timashev and S.~Nemirovski for helpful comments and suggestions.
We thank B.~Webster for useful discussions and information about \cite{KTWWY}.

\section{Preliminaries}\label{s2}

\subsection{Skew-group ring} \label{s2.1}
Throughout the paper we work over the field $\mathbb C$ of complex numbers. 
Let $G$ and $V$ be two complex-analytic or linear algebraic Lie groups such that 
$G$ acts on $V$. Let  $G \ltimes V$ be the corresponding semidirect product. 
To simplify notations we will write 
$G$ and $V$ for subgroups $G\times \{e\}$ and $\{e\}\times V$ in  $G \ltimes V$, 
respectively.  On $V$ we have a free transitive action of $V$ by left translations 
$\phi_{\xi}$, where $\xi\in V$, and an action of $G$ given by $v \mapsto g\cdot v = g v g^{-1}$. 
Both actions are assumed to be holomorphic or algebraic. Note that $e\in V$ is a 
fixed point of the action of $G$. We denote by  $\gimel$ a fixed subgroup in $V$.

Denote by $\mathbb C \gimel$ the group algebra of $\gimel$. 
Consider the vector space of global meromorphic (or rational) sections of 
the trivial vector bundle $V\times \mathbb C \gimel \to V$.  We will denote this vector space by 
$\mathcal S(\gimel)$ or simply by $\mathcal S$, if $\gimel$ is clear from the context. 
We assume that any section of $\mathcal S(\gimel)$ has the form 
$f=\sum\limits_{i=1}^s f_i \phi_{\xi_i}$, where $f_i$ are meromorphic (or rational) 
functions on $V$, $\xi_i\in \gimel$ and $s< \infty$.  The vector space $\mathcal S(\gimel)$ 
has the natural structure  of a skew-group ring defined in the following way:
\begin{align*}
\sum_{i} f_i \phi_{\xi_i} \circ \sum_{j} f'_j \phi_{\xi'_j} = \sum_{i,j} f_i \phi_{\xi_i}(f'_j) \phi_{\xi_i\circ \xi'_j}.
\end{align*}
Here, by definition, 
$\phi_{\xi_i}(f'_j)(x):= f'_j(\xi^{-1}_i(x))$ for any $x\in V$.
Clearly, for any subgroup $\gimel'\in\gimel$ the ring  $\mathcal S(\gimel')$ can be viewed as
a subring in $\mathcal S(\gimel)$ in the obvious way.
 
The action of $G$ on $V$ induces an action of $G$ on $\mathcal S(V)$ and also an
action of $G$ on $S(\gimel)$ provided that $\gimel$ is $G$-invariant. 
More precisely, $g\cdot f \phi_{\xi}= (g\cdot f) \phi_{g\xi g^{-1}}$ and 
$g\cdot f$ is a function on $V$ defined as follows $g\cdot f(v) = f (g^{-1}\cdot v)$ for $v\in V$.
Let $\gimel$ be $G$-invariant.  Then we have the subring $\mathcal S(\gimel)^{G}$ of 
$G$-invariant sections of $\mathcal S(\gimel)$. Denote by $\mathcal M$ and by $\mathcal O$ 
the sheaves of meromorphic (or rational) and holomorphic (or polynomial) functions on $V$, 
respectively. For any $v\in V$, we denote by $\mathcal M_v$ and  $\mathcal O_v$ 
the corresponding algebras of germs of meromorphic and holomorphic functions at $v$. We put 
$$
\mathfrak M := \bigoplus_{x\in V} \mathcal M_x, \quad \mathfrak O := \bigoplus_{x\in V} \mathcal O_x.
$$
If $W\subset V$ is a subset, we set $\mathfrak M|_{W} := \bigoplus\limits_{x\in W} \mathcal M_x$ and $\mathfrak O|_{W} := \bigoplus\limits_{x\in W} \mathcal O_x$.

The ring $\mathcal S(V)$ acts on the vector space $\mathfrak M$ in the following way:
$$
f \phi_{\xi}: \mathcal M_v \to \mathcal M_{\xi(v)}, \quad F_v \mapsto 
(f\phi_{\xi} (F_v))_{\xi(v)}.
$$
Consequently, the ring $\mathcal S(\gimel)$ acts on the vector space $\mathfrak M|_{\gimel \cdot v}$, 
where $v\in V$ and $\gimel \cdot v$ is the $\gimel$-orbit of $v$. 
Note that, in general, we do not have any action of $\mathcal S$ on $\mathfrak O$, 
since sections of $\mathcal S$ are assumed to be meromorphic (resp. rational) and not 
holomorphic (resp. polynomial).

In case we need to distinguish complex-analytic and algebraic categories, 
we will use the subscripts $\mathsf{C} $ and $\mathsf A$, respectively. 
For example, we will write $\mathcal S_{\mathsf{C}}$ and 
$\mathcal O_{\mathsf{C}}$ to specify that we are working with meromorphic 
sections of our skew-ring and with holomorphic functions on $V$.

\subsection{Example: the classical Gelfand-Zeitlin operators}\label{sec Gelfand-Zeitlin operators}\label{s2.2}

For $n\geq 2$, denote by $V$ the vector space 
$$
V= \mathbb C^{n(n+1)/2}=\{ (v_{ki})\,\, |\,\,  1\leq i \leq k \leq n  \}.
$$ 
An element of $V$ is called a {\it Gelfand-Zeitlin tableaux}.
Let $\gimel \simeq \mathbb Z^{n(n-1)/2}$ be the subgroup of $V$ 
generated by Kronecker vectors $\delta^{st}=(\delta^{st}_{ki})$, where $k$ and $i$ are as above,
$1\leq t \leq s \leq n-1$ and $\delta_{ki}^{st}=1$, 
if $k=s$ and $i=t$, and $\delta_{ki}^{st}=0$ otherwise.  
The product $G=S_1\times S_2\times \cdots \times S_n$ of symmetric groups acts 
on $V$ in the following way: the element $s=(s_1,\ldots, s_n)\in G$ acts on 
$v=(v_{ki})$ via $(s (v))_{ki} = v_{k s_k(i)}$. 
For $a\in \mathbb C$, set $\xi^a_k=a\delta^{k1}$.
Consider the following elements in $\mathcal S(\gimel)^G$:
\begin{gather*}
E_{k,k+1} = \sum_{g\in G}g \cdot \frac{\prod\limits_{j= 1}^{k+1} (v_{k1}- v_{k+1,j})}{\prod\limits_{j= 2}^k (v_{k1}- v_{kj})} \phi_{\xi^1_k};\quad
E_{k+1,k}   = \sum_{g\in G}g\cdot \frac{\prod\limits_{j= 1}^{k-1} (v_{k1}- v_{k-1,j})}{\prod\limits_{j = 2}^k (v_{k1}- v_{kj})} \phi_{\xi^1_k}^{-1};\\
E_{kk}   = \sum_{i=1}^k (v_{ki} +i-1) - \sum_{i=1}^{k-1} (v_{k-1,i} +i-1).
\end{gather*} 
The subalgebra $U\subset \mathcal S(\gimel)^G$ generated by $E_{st}$ is 
isomorphic to universal enveloping algebra $\mathcal U(\mathfrak{gl}_n(\mathbb C))$, 
see e.g. \cite{DFO,Ma} for details.

\subsection{Orthogonal Gelfand-Zeitlin algebras}\label{sect def of OGZA}\label{s2.3}

Orthogonal Gelfand-Zeitlin algebras are generalizations of 
$\mathcal U(\mathfrak{gl}_n(\mathbb C))$ introduced in \cite{Ma}. 
Fix a positive integer $n\geq 2$ and let $n_k$, where $k=1,\ldots, n$, be positive integers. 
Denote by $V$ the following  vector space 
$$
V= \mathbb C^{\sum_k n_k}=\{ v=(v_{ki})\,\, |\,\,  1\leq i \leq n_k, \,\, 1\leq k \leq  n   \}.
$$
Let $\gimel \simeq \mathbb Z^{\sum_k n_k}$ be the subgroup of $V$ generated by 
$\delta^{st}=(\delta^{st}_{ki})$, where $1\leq t \leq n_k$, $1\leq s \leq  n $,
as in Subsection~\ref{s2.2}. The group $G=S_{n_1}\times S_{n_2}\times \cdots \times S_{n_n}$
acts on $V$ as in Subsection~\ref{s2.2} which defines the
ring $\mathcal S$ and its subring $\mathcal S^G$. 

With $\xi^1_k$ defined as in Subsection~\ref{s2.2}, an {\it orthogonal Gelfand-Zeitlin algebra} 
is a subalgebra in $\mathcal S^G$ generated by all $G$-invariant polynomials on $V$ and by the 
elements
\begin{align*}
E_{k} = \sum_{g\in G}g \cdot \frac{\prod\limits_{j= 1}^{n_{k+1}} (v_{k1}- v_{k+1,j})}
{\prod\limits_{j= 2}^{n_k} (v_{k1}- v_{kj})} \phi_{\xi^1_k};\quad
F_{k}   = \sum_{g\in G}g \cdot\frac{\prod\limits_{j= 1}^{n_{k-1}} (v_{k1}- v_{k-1,j})}
{\prod\limits_{j= 2}^{n_k} (v_{k1}- v_{kj})} \phi_{\xi^1_k}^{-1}.
\end{align*}
The algebra $\mathcal U(\mathfrak{gl}_n(\mathbb C))$ from Subsection~\ref{s2.2} is just a special
case of this construction, for $n_k=k$. 

Note that the generators $E_k$ and $F_k$ of the orthogonal Gelfand-Zeitlin algebra are rational
(and not polynomial), however, 
it was shown in \cite[Proposition~1]{EMV} that the operators $E_k$ and $F_k$ preserve the vector space $H^0(V,\mathcal O^G)$.
By \cite{KTWWY}, orthogonal Gelfand-Zeitlin algebras are related to shifted Yangians and
generalized $W$-algebras in type $A$.

\subsection{Standard algebras of type $\mathbb A$}\label{sec Standard algebras of type A} \label{s2.4}

Let $V$ and $G$ be as in Section \ref{sect def of OGZA}. 
An element $A$ in $\mathcal S^G$ is called {\it standard} if
$A=\sum\limits_{g\in G}g\cdot (f\phi_{\xi^a_k})$, where $a\in \mathbb C$.

\begin{definition}\label{def standard algebra of type A}
A subalgebra $\mathcal A\subset \mathcal S(V)^G$ is called 
{\it standard of type $\mathbb A$} if $\mathcal A$ is generated by linear 
combinations of standard elements.
\end{definition}

Orthogonal Gelfand-Zeitlin algebras are examples of standard algebras of type $\mathbb A$. 
Other examples of such algebras are: finite W-algebras  of type $A$ and, more general, 
standard Galois orders of type $A$, see \cite[Section~8]{FuZ} or \cite{Har} for definition. 
In Section \ref{sec Structure theorem for standard algebras} we will show that 
standard algebras of type $\mathbb A$ that preserve the vector space  $\mathfrak O^G$  
are exactly standard Galois orders of type $A$. 

\subsection{Harish-Chandra modules} \label{s2.5}

In this paper we will study modules which fit into the general philosophy of {\em Harish-Chandra modules}.
Let $\mathcal A\subset \mathcal S^G$ be a subalgebra containing, as a subalgebra, the algebra $\mathcal B$  
of all global $G$-invariant functions on $V$. 

\begin{definition}
We say that an {\it $\mathcal A$-module $M$ is a  Harish-Chandra module} 
provided that the action of $\mathcal B$ on $M$ is locally finite.
\end{definition}

Gelfand-Zeitlin modules for orthogonal Gelfand-Zeitlin algebras and Galois orders, studied in 
\cite{EMV,FuZ,Vi,Vi2} are examples of Harish-Chandra modules.

\section{Algebras preserving the vector space $\mathfrak O^G$ and their modules}\label{s3}

\subsection{A fibration corresponding to the sheaf of invariant 
functions}\label{s3.1}

Consider a semidirect product $G \ltimes V$, where $G$ is a finite group and 
$V$ is a complex-analytic or linear algebraic group. As above we denote by 
$\mathcal O$ the structure sheaf of the complex-analytic  (or algebraic) 
variety $V$. In other words, we assume that all sections of $\mathcal O$ 
are holomorphic or polynomial functions on $V$, respectively. We now define 
the sheaf $\mathcal O^{G}$ of $G$-invariant holomorphic (or polynomial) 
functions on $V/G$. For a $G$-invariant open set $U$ in $V$,  we let 
$\mathcal O^{G}(U/G)$ be the algebra of $G$-invariant holomorphic 
(or polynomial) functions on $U$. Below we will consider the algebra 
$\mathcal O^{G}_v$ of germs of  functions at a point $v\in V$. 
By definition, $\mathcal O^{G}_v$ is the algebra of germs 
$f\in \mathcal O_v$ such that there exists a $G$-invariant function 
$F\in \mathcal O_{G\cdot v}$ that has the germ $f$ at the point $v$. 
We put
$$
\mathfrak M^G := \bigoplus_{\bar x\in V/G} \mathcal M^G_{\bar x}, \quad \mathfrak O^G := \bigoplus_{\bar x\in V/G} \mathcal O^G_{\bar x}.
$$
If $W\subset V$ is a $G$-invariant subset, we set $\mathfrak M^G|_{W} := \bigoplus\limits_{\bar x\in W/G} \mathcal M_{\bar x}$ and $\mathfrak O^G|_{W} := \bigoplus\limits_{\bar x\in W/G} \mathcal O_{\bar x}$.

If $v$ is a fixed point of the action of $G$, then the algebra 
$\mathcal O^{G}_v$ is invariant with respect to the action of $G$. 
The group $G$ has at least one fixed point, namely, the identity $e\in V$.   
Consider the algebra $\mathcal O_e$ of germs of functions at the 
point $e$ and its $G$-invariant subalgebra 
$\mathcal O^{G}_e \subset \mathcal O_e$. Denote by $\mathcal J_e$ 
the ideal in $\mathcal O_e$ generated by functions from 
$\mathcal O^{G}_e$ that are equal to $0$ at $e$.  As above, we denote 
by $\phi_{\xi}: \mathcal O_{x} \to \mathcal O_{\xi x}$, 
$f \mapsto \phi_{\xi}(f)= f\circ \xi^{-1}$, the left translation by $\xi\in V$. 

\begin{lemma}\label{lem raznesinie} 
Let $G$ be a finite group, $\xi\in V$,  $G_{\xi}$ 
the stabilizer of $\xi$ and $\xi G \xi^{-1}\subset G \ltimes V$ 
the group obtained from $G$ by conjugation with $\xi$.  We have
\begin{align*}
\phi_{\xi}(\mathcal O^{G_{\xi}}_e) = \mathcal O^{G}_{\xi}\quad\text{ and }\quad 
\phi_{\xi} (\mathcal O^{G}_e) = \mathcal O^{\xi G \xi^{-1}}_{\xi}.
\end{align*} 
In particular, we have
$$
\phi_{\xi}(\mathcal O^{G_{\xi}}_e /(\mathcal O^{G_{\xi}}_e \cap \mathcal J_e)) 
= \mathcal O^{G}_{\xi}/\langle \mathcal O^{G}_{\xi} \cap (\mathcal 
O^{\xi G \xi^{-1}}_{\xi})^+\rangle,
$$
where the superscript $+$ means that we consider all functions from 
$\mathcal O^{\xi G \xi^{-1}}_{\xi}$ that are equal to $0$ at $\xi$ and 
$\langle \mathcal O^{G}_{\xi} \cap (\mathcal O^{\xi G \xi^{-1}}_{\xi})^+\rangle$ 
denotes the ideal in $\mathcal O^{G}_{\xi}$ generated by $\mathcal O^{G}_{\xi} \cap (\mathcal O^{\xi G \xi^{-1}}_{\xi})^+$.
\end{lemma}

\begin{proof}  
First of all, we note that $\mathcal O^{G}_{\xi} = \mathcal O^{G_{\xi}}_{\xi}$.
Indeed, if $f\in \mathcal O^{G}_{\xi}$, then, clearly, 
$f\in \mathcal O^{G_{\xi}}_{\xi}$.  Further, if $f\in \mathcal O^{G_{\xi}}_{\xi}$,
then the sum of germs $\sum\limits_{g\in G} g(f)$ is an element of
$\bigoplus\limits_{g\in G} \mathcal O^{G}_{g \cdot \xi}$. Therefore 
$f\in \mathcal O^{G}_{\xi}$. Furthermore, the sheaf isomorphism 
$\phi_{\xi}:  \mathcal O_e \to \mathcal O_{\xi}$ is $G_{\xi}$-equivariant. 
Therefore, $\phi_{\xi}(\mathcal O^{G_{\xi}}_e) = \mathcal O^{G_{\xi}}_{\xi}$. 
The second and the third statements are clear, details are left to the reader. 
\end{proof}

The above defines the vector space $\mathcal O_e/ \mathcal J_e$ and its 
subspaces $\mathcal O^{G_{\xi}}_e /(\mathcal O^{G_{\xi}}_e \cap \mathcal J_e)$,
for any $\xi\in V$. Consider the following correspondence:
\begin{align*}
V\ni \xi \longmapsto \mathbb E_{\xi}:=  \mathcal O^{G}_{\xi}/\langle \mathcal O^{G}_{\xi} \cap (\mathcal O^{\xi G \xi^{-1}}_{\xi})^+\rangle = \phi_{\xi} (\mathcal O^{G_{\xi}}_e /(\mathcal O^{G_{\xi}}_e \cap \mathcal J_e)).
\end{align*}
This correspondence defines a fibration  
$\mathbb E= (\mathbb E_{\xi})_{\xi \in V}$ of vector spaces  over $V$. 
The group $G$ acts naturally on the fibration $\mathbb E$. 
Indeed, if $f\in \mathcal O^{G}_{\xi}$, then
$g\cdot f \in \mathcal O^{G}_{g\cdot\xi}$, and if 
$f\in (\mathcal O^{\xi G \xi^{-1}}_{\xi})^+$, 
then $g \cdot f\in (\mathcal O^{g \cdot\xi G (g \cdot\xi)^{-1}}_{g \cdot\xi})^+$. 
Now we can define the fibration 
$\mathbb E^G= (\mathbb E^G_{\bar\xi})_{\bar\xi \in V/G}$ 
on $V/G$ in the following way: $\mathbb E^G_{\bar\xi}$ is the vector space of 
all $G$-invariant elements from 
$\bigoplus\limits_{\xi'\in \bar\xi} \mathbb {E}_{\xi'}$. We set
$$
\mathfrak E := \bigoplus_{x\in V}  \mathbb E_{x}, \quad  \mathfrak E|_{W'} := \bigoplus_{x\in W'}  \mathbb E_{x}, \quad \mathfrak E^G := \bigoplus_{\bar x\in V/G}  \mathbb E^G_{\bar x}, \quad \mathfrak E^G|_{W} := \bigoplus_{\bar x\in W/G}  \mathbb E^G_{\bar x}
$$
for a subset $W'\subset V$ and for a $G$-invariant subset $W\subset V$.

\subsection{Action of elements in $\mathcal S$ on $\mathfrak E$}\label{s3.2}
 
Let  $\gimel$ be a $G$-invariant subgroup in $V$ and $v\in V$. 
Consider the $G$-invariant subset  $(G \ltimes \gimel) \cdot v$. 
Then $\mathfrak E^G|_{(G \ltimes \gimel) \cdot v/G}$ is defined.  
Take an element $A$ in $\mathcal S$ preserving 
the vector space $\mathfrak O^{G}|_{(G \ltimes \gimel) \cdot v/G}$. 
Any such $A$ has the following form on 
$\mathfrak O^{G}|_{(G \ltimes \gimel) \cdot v/G}$:
\begin{equation}\label{eq operator A}
A|_{\mathfrak O^{G}|_{(G \ltimes \gimel) \cdot v/G}}= 
\sum_i \sum\limits_{h\in G} h \cdot (f_i \phi_{\xi_i}). 
\end{equation}
Note that $A$ may be meromorphic. Also, we do not assume that
$A(\mathfrak O)\subset \mathfrak O$. 

\begin{theorem}\label{theor action in the bundle} 
Assume that $G$ is a finite group and $A$ sends  
$\mathfrak O^{G}|_{(G \ltimes \gimel) \cdot v/G}$ to itself. 
Then the action of $A$ on
$\mathfrak O^{G}|_{(G \ltimes \gimel) \cdot v/G}$ 
induces an action of $A$ on $\mathfrak E^G|_{(G \ltimes \gimel) \cdot v/G}$.
\end{theorem}

\begin{proof} 
The element $A\in \mathcal S$ acts on 
$\mathfrak O^{G}|_{(G \ltimes \gimel) \cdot v/G}$.  We need to show that this action 
induces an  action on $\mathfrak E^G|_{(G \ltimes \gimel) \cdot v/G}$, 
or, equivalently, that it induces an  action on the vector space 
$$
\bigoplus\limits_{\bar \xi\in (G \ltimes \gimel) \cdot v/G} 
\big[\bigoplus\limits_{\xi'\in \bar\xi} \mathcal O^{G}_{ \xi'} /  \phi_{\xi'} 
(\mathcal O^{G_{\xi'}}_e \cap \mathcal J_e)\big]^G.
$$ 
In other words, we need to show that $A(F)$, where 
$
F\in [\bigoplus\limits_{g\in G} \phi_{g \cdot \xi} 
(\mathcal O^{G_{g \cdot \xi}}_e \cap \mathcal J_e)]^G,
$ 
is a sum of elements from $[\bigoplus\limits_{g\in G} \phi_{g \cdot \xi'} 
(\mathcal O^{G_{g \cdot \xi'}}_e \cap \mathcal J_e)]^G$ for various $\xi'$.
Let us take $F$ such that there exists $F'\in \mathcal O^{G}_{e}$ and $X \in \mathcal O^{G_{g \cdot \xi}}_e$ with 
$$
F= \sum\limits_{g\in G} (F'\circ g \cdot \xi^{-1}) [g \cdot (X \circ \xi^{-1})].
$$ 
Note that $F'$ is 
either in the ideal $\mathcal J_e$ or is an invertible $G$-invariant 
element. We have
\begin{align*}
A(F) =   \sum_i\sum_{h,g\in G} (h\cdot f_i) F' \circ (g \cdot \xi^{-1} \circ h \cdot \xi_i^{-1}) [g \cdot (X) \circ g \cdot \xi^{-1} \circ h \cdot \xi_i^{-1}] .
\end{align*}
This is a sum of $G$-invariant germs supported at the points 
$h\cdot \xi_i \circ g \cdot \xi$. Consider, for example, the germ of 
$A(F)$ at the point $\eta:=h_0\cdot \xi_{i_0} \circ \xi$:
\begin{equation}\label{eq A(F)_eta}
\begin{split}
A(F)_{\eta} = \sum_{(g,h,i)\in\Lambda} (h\cdot f_i) 
F' \circ (g \cdot \xi^{-1} \circ h \cdot \xi_i^{-1}) [g \cdot (X) \circ \eta^{-1}] = \\
F' \circ \eta^{-1} \sum_{(g,h,i)\in\Lambda} (h\cdot f_i) [g \cdot (X) \circ \eta^{-1}],
\end{split}
\end{equation}
where $\Lambda = \{(g,h,i)\,|\,\, (h\cdot \xi_i) \circ (g \cdot \xi) = \eta\}$. 
We see that the product of a meromorphic function 
$$
H:=\sum\limits_{(g,h,i)\in\Lambda} (h\cdot f_i) g \cdot (X) \circ \eta^{-1}
$$ 
and a holomorphic function 
$F' \circ \eta^{-1}$ is holomorphic, since $A(F)_{\eta}$ is holomorphic. 
This holds for any $F'\in  \mathcal O^{G}_{e}$, in particular, for constant $F'$. The latter implies that $H$ is holomorphic at $\eta$. Similarly, we conclude 
that $H$ is in $\mathcal O^{G}_{\eta}$.  Summing up, we have  
$F' \circ \eta^{-1} \in \mathcal O^{\eta G \eta^{-1}}_{\eta} $ and 
$H \in \mathcal O^{G}_{\eta}$. Note that, from $F'\in \mathcal J_e$, 
it follows that $F' \circ \eta^{-1} \in (\mathcal O^{\eta G \eta^{-1}}_{\eta})^+$.
Now the assertion of the theorem follows from Lemma~\ref{lem raznesinie}. 
\end{proof}

\subsection{$\mathcal A$-modules corresponding to $\mathfrak E$}\label{s3.3}

For convenience we put
\begin{equation}\label{eq vector space sum E_xi_i}
M(G,(G \ltimes \gimel) \cdot v) = \mathfrak E|_{\bar\xi \in (G \ltimes \gimel) \cdot v/G}.
\end{equation}
Denote by $M^*(G,(G \ltimes \gimel) \cdot v)$ the vector space 
\begin{equation}\label{eq vector space sum E^*_xi_i}
M^*(G,(G \ltimes \gimel) \cdot v): = 
\bigoplus_{\bar\xi \in (G \ltimes \gimel) \cdot v/G} (\mathbb E^G_{\bar\xi})^*.
\end{equation}
Note that, in general, 
$M^*(G,(G \ltimes \gimel) \cdot v)\subsetneq 
\big(\mathfrak E|_{\bar\xi \in (G \ltimes \gimel) \cdot v/G}\big)^*$. We will need the following lemma.

\begin{lemma}\label{lem action on M^*}
Assume that $A$ sends the vector space $\mathfrak O^{G}|_{(G \ltimes \gimel) \cdot v/G}$ 
to itself. Then the action of $A$ on 
$\mathfrak O^{G}|_{(G \ltimes \gimel) \cdot v/G}$ induces 
an action of $A$ on $M^*(G,(G \ltimes \gimel) \cdot v)$.
\end{lemma}

\begin{proof} 
Let us take $A= \sum\limits_i \sum\limits_{h\in G} h \cdot (f_i \phi_{\xi_i})$,
$\alpha\in (\mathbb E^G_{\bar\eta})^*$ and 
$\sum\limits_{g\in G}g\cdot F\in \mathcal O_{\bar \xi}$. We have
\begin{align*}
[A(\alpha)]\big(\sum\limits_{g\in G}g\cdot F\big) = 
\alpha \big( \sum\limits_i \sum\limits_{h,g\in G} h \cdot (f_i) 
(g\cdot F)\circ  h \cdot\xi_i^{-1} \big).
\end{align*}
Inside the brackets on the right hand side we have a sum of 
$G$-invariant germs supported at the points from the  finite set
$\{h \cdot\xi_i\circ g\cdot\xi\,\, |\,\, g,h\in G \}$. Therefore,
$[A(\alpha)]\big(\sum\limits_{g\in G}g\cdot F\big) = 0$, if 
$\bar\eta \notin \{h \cdot\xi_i\circ g\cdot\xi\,\, |\,\, g,h\in G \}/G$. 
In other words, 
$$
A(\alpha)\subset \bigoplus_{\bar\xi' \in \Lambda/G} (\mathbb E^G_{\bar\xi})^*, \quad \text{where}\quad \Lambda = \{h \cdot\xi^{-1}_i\circ g\cdot\eta\,\, |\,\, g,h\in G \}
$$
and the proof is complete.
\end{proof}

As a consequence of Theorem~\ref{theor action in the bundle} 
and Lemma~\ref{lem action on M^*}, we have the following statement. 

\begin{corollary}\label{cor M is a A-module}
Let $\mathcal A$ be a subalgebra in $\mathcal S(\gimel)$ that preserves 
the vector space $\mathfrak O^{G}|_{(G \ltimes \gimel) \cdot v}$. 
Then both $M(G,(G \ltimes \gimel) \cdot v)$ and $M^*(G,(G \ltimes \gimel) \cdot v)$
are $\mathcal A$-modules. 
\end{corollary}

In the next sections we will consider the case when $G$ acts locally as a 
reflection group. In this case all vector spaces $\mathbb E^G_{\bar\xi}$ are 
finite dimensional of dimension $|G|$ by Chevalley-Shephard-Todd Theorem. 

\subsection{Construction of new $\mathcal A$-modules }\label{s3.4}

Recall that $\gimel$ is a $G$-invariant subgroup in $V$. Let $\mathcal A$ 
be a subalgebra in $\mathcal S(\gimel)$ that preserves the vector space $\mathfrak O^{G}|_{(G \ltimes \gimel) \cdot v/G}$, where $v\in V$ is a 
fixed point.  Denote by $G_{\gimel \cdot v}$ the stabilizer in $G$ of the 
orbit $\gimel \cdot v$. Let 
$W:= (G \ltimes \gimel) \cdot v\setminus \gimel \cdot v$. In other words, 
$W\subset V$ is the union of all orbits of $\gimel$ in  
$(G \ltimes \gimel) \cdot v$ except for $\gimel \cdot v$. By definition, 
the group $G_{\gimel \cdot v}$ acts on $\gimel \cdot v$. Therefore,  
$G_{\gimel \cdot v}$ acts on $W$ too. 

Further, we have a natural projection 
$\pi_G: \mathfrak O^{G}|_{(G \ltimes \gimel) \cdot v/G} \to  
\mathfrak O^{G_{\gimel \cdot v}}|_{\gimel \cdot v/G_{\gimel \cdot v}}$ 
defined by the following formula:
\begin{equation}\label{eq def of pr}
\mathfrak O^{G}|_{(G \ltimes \gimel) \cdot v/G}\ni F= \sum_{g\in G_{\gimel \cdot v}} g\cdot f_{\xi} + \sum_{g\in L} g\cdot f_{\xi} \longmapsto \sum_{g\in G_{\gimel \cdot v}} g\cdot f_{\xi}\in \mathfrak O^{G_{\gimel \cdot v}}|_{\gimel \cdot v/G_{\gimel \cdot v}}, 
\end{equation}
where $\xi\in \gimel \cdot v$, $f_{\xi}\in \mathcal O^{G_{\xi}}_{\xi}$ is a $G_{\xi}$-invariant germ and 
$$
L:= G\setminus G_{\gimel \cdot v} =  
\{g\in G\,\,|\,\, g\cdot \xi \notin \gimel \cdot v \}.
$$ 
Note that, for any such $F$, there exists $f_{\xi}$ with $\xi\in \gimel \cdot v$
and the map \eqref{eq def of pr} is independent of the choice 
of $\xi\in \gimel \cdot v$.  

\begin{lemma}
The map $\pi_G$ is a bijection.
\end{lemma}

\begin{proof}
Assume that $\pi_G(F)=0$. Then $f_{\xi}=0$ and hence 
$F':=\sum\limits_{g\in G_{\gimel \cdot v}} g\cdot f_{\xi}=0$. Further, 
let us take 
$F'=\sum\limits_{g\in G_{\gimel \cdot v}} g\cdot 
f_{\xi}\in \mathcal O^{G_{\gimel \cdot v}}|_{\gimel \cdot v/G_{\gimel \cdot v}}$. Then 
$$
F'= \pi_G ( \sum_{g\in G_{\gimel \cdot v}} g\cdot f_{\xi} + 
\sum_{g\in L} g\cdot f_{\xi}). 
$$
Explicitly, the map $\pi_G^{-1}$ is given by 
$$
\pi_G^{-1} (\sum_{g\in G_{\gimel \cdot v}} g\cdot f_{\xi}) = 
\frac{1}{|G_{\gimel \cdot v}|} \sum_{g'\in G} 
g'\cdot (\sum_{g\in G_{\gimel \cdot v}} g\cdot f_{\xi}). 
$$
\end{proof}

We will need the following proposition. 
 
\begin{proposition}\label{prop isom pr} 
Let $\mathcal A$ be a subalgebra in $\mathcal S(\gimel)$ that 
preserves the vector space   $\mathfrak O^{G}|_{(G \ltimes \gimel) \cdot v/G}$. 
Then $\mathcal A$ also preserves 
$\mathfrak O^{G_{\gimel \cdot v}}|_{\gimel \cdot v/G_{\gimel \cdot v}} $ 
and the map $\pi$ is an isomorphism of $\mathcal A$-modules. 
\end{proposition}

\begin{proof}
Let $A\in \mathcal A$ be as in \eqref{eq operator A}. 
We apply $A$ to a germ 
$F=\sum\limits_{g\in G} g\cdot f_{\xi} \in \mathcal O^G_{\bar \xi}$, 
where $f_{\xi}\in \mathcal O^{G_{\xi}}_{\xi}$, $\bar \xi = G\cdot \xi$ 
and $\xi\in \gimel \cdot v$. We get
\begin{align*}
A(F) =   \sum_i\sum_{h,g\in G} (h\cdot f_i) [(g\cdot f_{\xi}) \circ  h \cdot \xi_i^{-1}].
\end{align*}
Note that, if $(h\cdot \xi_i) \circ (g \cdot \xi)\in \gimel \cdot v$, then 
$g \cdot \xi  \in \gimel \cdot v$ and hence $g\in G_{\gimel \cdot v}$. 

Now we compute the germ of $A(F)$ at the point 
$\eta:=(h_0\cdot \xi_{i_0}) \circ (g_0 \cdot\xi)\in \gimel \cdot v$:
\begin{align*}
A(F)_{\eta} = \sum_{(g,h,i)\in\Lambda} (h\cdot f_i) [(g\cdot f_{\xi}) 
\circ  h \cdot \xi_i^{-1}],
\end{align*}
where $\Lambda = \{(g,h,i)\,|\,\, (h\cdot \xi_i) \circ (g \cdot \xi) = \eta\}$. 
If $(h\cdot \xi_i) \circ (g \cdot \xi)=\eta$, then
$(g, h, i) \in \Lambda $ and we have 
$g \cdot \xi =\phi_{h\cdot \xi_i^{-1}}  (\eta)$ 
implying $g\in G_{\gimel \cdot v}$. As a consequence of these 
observations, we obtain
$$
\pi(A(F)) = A(\pi(F)).
$$
In particular, this equality implies that $A(\pi(F))$  is holomorphic and
therefore $A$ preserves $\mathfrak O^{G_{\gimel \cdot v}}|_{\gimel \cdot v/G_{\gimel \cdot v}} $. 
It also implies that $\pi$ is a homomorphism of $\mathcal A$-modules. 
The proof is complete.
\end{proof}

Here comes yet another construction of $\mathcal A$-modules. 
Let $\mathcal A$ be as above and $H$ a subgroup of $G$ such that  
$\mathcal A$ preserves the vector space $\mathfrak O^{H}|_{(H \ltimes \gimel) \cdot v/H}$.
For the pair $H\subset G$, we have the obvious inclusion 
$P^G_H: \mathcal O^{G}|_{(G \ltimes \gimel) \cdot v/G} \hookrightarrow 
\mathcal O^{H}|_{(G \ltimes \gimel) \cdot v/H}$.

\begin{lemma}\label{lem commut diagr} 
Assume that $\mathcal A$ preserves the vector spaces 
$\mathfrak O^{G}|_{(G \ltimes \gimel) \cdot v/G}$ and  
$\mathfrak O^{H}|_{(H \ltimes \gimel) \cdot v/H}$. 
Then the diagram
$$
\xymatrix{
\mathfrak O^{G}|_{(G \ltimes \gimel) \cdot v/G} 
\ar[r]^{\pi_G} \ar[d]_{P^G_H}
& \mathfrak O^{G_{\gimel \cdot v}}|_{\gimel \cdot v/G_{\gimel \cdot v}} 
\ar[d]^{P^{G_{\gimel \cdot v}}_{H_{\gimel \cdot v}}} \\
\mathfrak O^{H}|_{(G \ltimes \gimel) \cdot v/G} \ar[r]^{\pi_H}&
\mathfrak O^{H_{\gimel \cdot v}}|_{\gimel \cdot v/H_{\gimel \cdot v}},
}
$$
in which all maps are homomorphisms of $\mathcal A$-modules, commutes. 
\end{lemma}

\begin{proof}
This follows directly from the definitions. 
\end{proof}

The above leads us to the following theorem.

\begin{theorem}\label{theor homomorphism of A-modules}
Assume that $\mathcal A$ preserves the vector spaces 
$\mathfrak O^{G}|_{(G \ltimes \gimel) \cdot v/G}$ and  
$\mathfrak O^{H}|_{(H \ltimes \gimel) \cdot v/H}$. 
Then we have the following commutative diagram of $\mathcal A$-modules:
$$
\xymatrix{
M(G,(G \ltimes \gimel) \cdot v) 
\ar[r]^{\tilde \pi_G} 
\ar[d]_{{\bf P}^G_H}
& M(G_{\gimel \cdot v},\gimel \cdot v) 
\ar[d]^{{\bf P}^{G_{\gimel \cdot v}}_{H_{\gimel \cdot v}}} \\
M(H,(G \ltimes \gimel) \cdot v) \ar[r]^{\tilde \pi_H}
& M(H_{\gimel \cdot v},\gimel \cdot v),
}
$$
where $\tilde{\pi}_G$ and $\tilde{\pi}_H$ are induced by $\pi_G$ and 
$\pi_H$ from Proposition \ref{prop isom pr}, respectively. 
Moreover, the map 
$$
\Upsilon = {\bf P}^G_H \circ \pi_G^{-1}: 
\mathfrak O^{G_{\gimel \cdot v}}|_{\gimel \cdot v/G_{\gimel \cdot v}} 
\longrightarrow  M(H,(G \ltimes \gimel) \cdot v)
$$ 
is also a homomorphism of $\mathcal A$-modules. 
\end{theorem}

\begin{proof}
Theorem \ref{theor action in the bundle} defines all involved
$\mathcal A$-module structures. Let us argue, for example, that the 
morphism $\tilde{\pi}_G$ of $\mathcal A$-modules induced by $\pi_G$ is 
well-defined. This follows from the fact that, to obtain the module 
$M(G,(G \ltimes \gimel) \cdot v)$, we factor out by the ideal generated 
by $G$-invariants and, to obtain the module 
$M(G_{\gimel \cdot v},\gimel \cdot v)$, we factor out by the ideal 
generated by $G_{\gimel \cdot v}$-invariants. As we obviously have 
$G_{\gimel \cdot v} \subset G$, the necessary statement is obtained
by the standard factorization argument.  The commutativity of the dia\-gram 
follows from   Lemma~\ref{lem commut diagr}. 
\end{proof}

\subsection{The vector space 
$(\mathcal O_{\mathsf{C}}/ \mathcal J_{\mathsf{C}})_e$ is finite dimensional} 
\label{s3.5}

In this section we show that the vector space
$(\mathcal O_{\mathsf{C}}/ \mathcal J_{\mathsf{C}})_e$ is finite dimensional.
In particular, this implies that the fibration $\mathbb E$ has 
finite dimensional fibers. Several observations of this section were
pointed out to us by D.~Timashev. 
 
Let $V$ be a complex-analytic or linear algebraic Lie group. 
Any linear algebraic group is a complex-analytic Lie group, see 
\cite{Hum}. Recall that we emphasize by the subscripts $\mathsf{C}$ and 
$\mathsf{A}$ objects in the complex-analytic and the algebraic category,
respectively. For example,  we denote by $\mathcal O_{\mathsf{C}}$ and by 
$\mathcal O_{\mathsf{A}}$ the sheaves of complex-analytic (holomorphic) 
and algebraic (polynomial) functions, respectively. 

Let $V$ be a linear algebraic group. Note that we can choose coordinates 
$(x_i)$ in a neighborhood $U$ of the identity  $e\in V$ such that $e$ 
is the origin and the vector space $W= \langle x_1,\ldots, x_n \rangle$ 
is $G$-invariant. Indeed, denote by $\mathfrak m_e$ the maximal ideal 
in $(\mathcal O_{\mathsf{A}})_e$. Then $\mathfrak m_e^2$ is a $G$-invariant
subspace in $\mathfrak m_e$. We choose any coordinates $\{y_1,\ldots, y_n\}$ 
in  $U$. Let $W'$ be the $\mathbb C$-span of 
$\{ g\cdot y_i\,\,|\,\, i=1,\ldots, n, g\in G \}$.  Then $W'$ and 
$W'\cap \mathfrak m_e^2$ are $G$-invariant. Since $G$ is finite, 
there exists $G$-invariant subspace $W$ such that 
$W' = W\oplus (W'\cap \mathfrak m_e^2)$. Let $x_1,\ldots, x_n$ be a 
basis in  $W$. If 
$f\in (\mathcal O_{\mathsf{C}}^G)_e$, then there exists a decomposition
$f=\sum\limits_{k=0}^{\infty}f_k$, where $f_k$ are $G$-invariant homogeneous
polynomials in $(x_i)$ of degree $k$. If $V$ is complex analytic but not 
algebraic, we mean by $(\mathcal O_{\mathsf{A}})_e$ the algebra of germs 
of polynomial functions in $(x_i)$. 

A classical fact from the invariant theory is that the extension 
$(\mathcal O_{\mathsf{A}}^G)_e\subset (\mathcal O_{\mathsf{A}})_e$ 
of rings is integral. Indeed, any polynomial $f\in (\mathcal O_{\mathsf{A}})_e$ 
is integral over $(\mathcal O_{\mathsf{A}}^G)_e$ since it is a root of 
the polynomial $\prod\limits_{g\in G} (t- g\cdot f)$. In particular, $f^{|G|}$ 
is a linear combination of $f^p$, where $p< |G|$, with coefficients from
$(\mathcal O_{\mathsf{A}}^G)_e$. 

\begin{lemma}
We have that $(\mathcal O_{\mathsf{A}})_e$ is a finitely generated 
$(\mathcal O_{\mathsf{A}}^G)_e$-module and the minimal 
number of generators is less than or equal to $|G|^{\dim V}$. 
\end{lemma}

\begin{proof}
The proof follows from the fact that $x_i^{|G|}$ is a linear combination 
of $x_i^p$, where $p< |G|$, with coefficients from $(\mathcal O_{\mathsf{A}}^G)_e$. 
\end{proof}

\begin{corollary}
The vector space $(\mathcal O_{\mathsf{A}}/ \mathcal J_{\mathsf{A}})_e$ 
is finite dimensional and its dimension is less than or equal to $|G|^{\dim V}$. 
\end{corollary}

\begin{theorem}\label{cor dim of algebraic anf holom coinsiede}
Let $V$ be a complex analytic or linear algebraic group and $G$ a 
finite group acting on $V$. Then 
$$
(\mathcal O_{\mathsf{A}}/ \mathcal J_{\mathsf{A}})_e \simeq 
(\mathcal O_{\mathsf{C}}/ \mathcal J_{\mathsf{C}})_e.
$$ 
In particular,  $(\mathcal O_{\mathsf{C}}/ \mathcal J_{\mathsf{C}})_e$ 
is finite dimensional and its dimension is less than or equal to $|G|^{\dim V}$. 
\end{theorem}

\begin{proof} 
We have the obvious map
\begin{equation}\label{eq isom of O/J}
(\mathcal O_{\mathsf{A}}/ \mathcal J_{\mathsf{A}})_e \longrightarrow 
(\mathcal O_{\mathsf{C}}/ \mathcal J_{\mathsf{C}})_e, \quad f \mapsto f+ 
(\mathcal J_{\mathsf{C}})_e. 
\end{equation}
Let us show that this map is a bijection. 

\smallskip

\noindent{\it Step 1.} 
Let us first show that the map \eqref{eq isom of O/J} is injective. 
To start with, 
assume that $f\in (\mathcal O_{\mathsf{A}})_e\cap (\mathcal J_{\mathsf{C}})_e$. 
Then $f=  \sum\limits_{j=1}^s f_{1j} f_{2j}$, where 
$f_{1j}=\sum\limits_{k=0}^{\infty} f^{j1}_k\in (\mathcal O_{\mathsf{C}})_e$, 
$f_{2j} = \sum\limits_{p=1}^{\infty} f^{j2}_p\in (\mathcal O_{\mathsf{C}}^G)_e$,
$f^{j1}_k$ are homogeneous polynomials in $(x_i)$ of degree $k$ and $f^{j2}_p$ 
are homogeneous $G$-invariant polynomials in $(x_i)$ of degree $p$. 
We see that the polynomial 
$f=\sum\limits_{j=1}^s \sum\limits_{k=0}^{\infty}  
\sum\limits_{p=1}^{\infty} f^{j1}_k f^{j2}_p$ is an element in 
$ (\mathcal J_{\mathsf{A}})_e$.

\smallskip
\noindent{\it Step 2.} Let us now show that the map (\ref{eq isom of O/J}) is surjective. Denote by $z_1, \dots, z_p$ a system of generators for the
$(\mathcal O_{\mathsf{A}}^G)_e$-module $(\mathcal O_{\mathsf{A}})_e$ and set
$N=\max\limits_{s} \{\deg z_s \}$. Let us take $f\in \mathfrak m^{N+1}_e$, 
where $\mathfrak m_e$ is the maximal ideal in $(\mathcal O_{\mathsf{C}})_e$.

Assume first that $f= \sum\limits_{i=N+1}^tf_i$, where $f_i$ is a 
homogeneous polynomial of degree $i$, is a polynomial. The polynomial $\prod\limits_{g\in G} (t- g\cdot f)$,  considered above, 
is homogeneous. Hence we can assume that $z_j$ are homogeneous and 
$f_i = \sum\limits_j f_{ij} z_j$ is a decomposition with homogeneous 
$G$-invariant coefficients. Since $\deg f_i>N$, we conclude that 
$f\in (\mathcal J_{\mathsf{C}})_e$.

Further, let us take 
$f=\sum\limits_{i=N+1}^{\infty}f_i\in (\mathcal O_{\mathsf{C}})_e$, 
where $f_i$ are homogeneous polynomials in $(x_i)$ of degree $i$. 
Assume that $f$ is not identically equal to zero on the $x_n$-axis 
(we may ensure this by a linear change of coordinates).
By  the Weierstrass preparation theorem, we have $f= P f_1$, where 
$P= x_n^r+ a_{r-1} x_n^{r-1} + \cdots +  a_{1} x_n + a_{0}$ 
is a Weierstrass polynomial and $f_1$ is a unit. Here $a_i$ is a 
holomorphic function in $x_1,\ldots, x_{n-1}$, for any $i$. 
Since $f_1$ is a unit, $P = f f_1^{-1}\in \mathfrak m^{N+1}_e$. 
Note that in the Taylor expansions of $a_{\alpha} x_n^{\alpha}$ and 
$a_{\beta} x_n^{\beta}$ in a neighborhood of $e$, where $\alpha\ne \beta$, 
we do not have equal summands. Therefore, 
$a_{\alpha} x_n^{\alpha}\in  \mathfrak m^{N+1}_e$ for any $\alpha$. 
Similarly, we apply the Weierstrass preparation theorem  to $a_{\alpha}$ 
and proceed inductively.  We obtain a polynomial in $\mathfrak m^{N+1}_e$ 
that, by the above, belongs to $(\mathcal J_{\mathsf{C}})_e$. Now, 
assume, by induction, that 
$a_{\alpha} x_n^{\alpha} \in (\mathcal J_{\mathsf{C}})_e$. Hence  
$P\in (\mathcal J_{\mathsf{C}})_e$ and therefore 
$f= P f_1 \in  (\mathcal J_{\mathsf{C}})_e$. 
 
Now we can show that the map \eqref{eq isom of O/J} is surjective. 
Indeed, by the above, any  element
$F\in (\mathcal O_{\mathsf{C}}/ \mathcal J_{\mathsf{C}})_e$ 
has a polynomial representative. This completes the proof. 
\end{proof}

Let $\mathcal A$ be as in Theorem \ref{theor homomorphism of A-modules} and $\mathcal B\subset \mathcal S$ be the algebra of $G$-invariant functions. 
Assume, in addition, that $\mathcal B\subset \mathcal A$  

\begin{proposition}\label{prop all modules are GZ}
The $\mathcal A$-modules constructed in 
Theorem~\ref{theor homomorphism of A-modules} are Harish-Chandra modules. 
\end{proposition}

\begin{proof}
This follows from~Theorem \ref{cor dim of algebraic anf holom coinsiede}.  
\end{proof}

\section{Rational Galois orders and their modules}\label{s4}

\subsection{Reflection groups and divided difference operators}
\label{sec reflection group, roots, div diff}\label{s4.1}

Let $V_{\mathbb R}$ be a vector space over $\mathbb R$ equipped with 
a non degenerate symmetric bilinear form $(\,,)$. Set 
$V= \mathbb C\otimes_{\mathbb R} V_{\mathbb R}$ and denote the corresponding to $(\,,)$ 
inner product on $V$ by the same symbol. For $v\in V$, the {\em reflection 
$\sigma_v$ with respect to $v$} is the linear transformation of $V$ that fixes the 
hyperplane $\{ w\in V\,\,|\,\, (w,v)=0\}$ and maps $v$ to $-v$. It is given by the formula 
$\sigma_v (x) = x- \frac{2(x,v)}{(v,v)}v$. A {\it root system} $\Phi$ 
is a finite subset in $V_{\mathbb R}\setminus\{0\}$ that satisfies the following properties:
\begin{itemize}
\item[(I)] If $x,y\in \Phi$, then $\sigma_{x}(y) \in \Phi$. 
\item[(II)] If $x$ and $k x$ in $\Phi$, for some $k\in\mathbb{R}$, then $k=\pm 1$.
\end{itemize}   
For a root system $\Phi$, the corresponding {\it reflection group} 
$G\subset GL(V)$ is the group generated by all reflections $\alpha_v$,
where $v\in \Phi$. A {\it system of simple roots} or a {\it basis} of 
$\Phi$ is a linearly independent subset in $\Phi$ such that every 
$x\in \Phi$ can be written as a linear combination of elements from $\Psi$ 
with all non-negative or all non-positive coefficients.  
Any root system $\Phi$ has a basis.  If a basis $\Psi\in \Phi$ is
fixed, we get a partition $\Phi = \Phi^+\cup \Phi^-$, where $\Phi^+$ 
is the system of positive roots and $\Phi^-$ is  the system of negative ones. 
Here a root $x$ is called {\it positive} (resp. {\it negative}) with 
respect to $\Psi$, if it is a linear combination of vectors from $\Psi$ 
with all non-negative (resp. non-positive) coefficients. We denote by $\Theta$
the set of  {\it simple reflections}, that is reflections corresponding to elements in $\Psi$.

Let $G$ be a reflection group, $\Psi$  be a system of simple roots and $\Theta$ 
be the corresponding system of simple reflections. For any $x\in V$, we have a unique 
$\gamma_x\in V^*$ such that $\gamma_x(y)= (x,y)$, for all $y\in V$. Further, for any 
simple reflection $\sigma_x \in \Theta$, we define the corresponding 
{\em divided difference operator} $\partial_{\sigma_x}$ on the set of
holomorphic (or meromorphic, or rational or polynomial) functions  on $V$ via
$$
\partial_{\sigma_x}\cdot  f := \frac{f - \sigma_x \cdot f }{\gamma_{x}}. 
$$  
For any $w\in G$, we set 
$\partial_w = \partial_{\sigma_1} \circ \cdots \circ \partial_{\sigma_p}$, 
where $w =  \sigma_1 \circ \cdots \circ \sigma_p $ is a reduced expression.  
By \cite[Page~5]{BGG}, we have $\partial_w =0$, if the expression $w =  \sigma_1 \circ \cdots \circ \sigma_p $ 
is not reduced. Moreover, the operator  $\partial_w$ is independent of the choice of a reduced expression.

\subsection{Rational Galois orders}\label{sec universal ring}\label{s4.2} 

Rational Galois orders is a large class of algebras introduced in 
\cite[Section~4]{Har}. This class includes, for instance, orthogonal 
Gelfand-Zeitlin algebras, finite W-algebras of type $A$ and, as we 
will see in Section~\ref{sec Structure theorem for standard algebras}, 
standard algebras of type $\mathbb A$ that preserve the vector space $\mathfrak O^G$. 
Note that a particular case of rational Galois orders was considered earlier 
in \cite{Vi,Vi2}. In the terminology of \cite{Vi,Vi2}, 
these are finitely generated over $H^0(V, \mathcal O^G)$ 
subalgebras in the so-called {\it universal ring}.

Let $G$ be a reflection group in $V$ as in 
Subsection~\ref{sec reflection group, roots, div diff} 
(note that the definition of a rational Galois order was given in 
\cite{Har} for a more general case of a pseudo-reflection group or a complex reflection group $G$). 
Let $\chi:G\to \mathbb C^{\times}$ be a character. The space of relative invariants 
$$
H^0(V,\mathcal O)^G_{\chi} := \{ f\in H^0(V,\mathcal O) \,\,| \,\, g\cdot f 
= \chi(g) f\,\,\, \text{for all}\, \,\, g\in G \}
$$
is, naturally, an  $H^0(V,\mathcal O)^G$-module. This module is free of rank
$1$ and is generated by 
$$
d_{\chi} = \prod_{H\in A(G)} (\gamma_H)^{a_H},
$$
where $A(G)$ is the set of all hyperplanes $H$ that are fixed by a 
certain element $\sigma_H$ in $G$, $\gamma_H\in V^*$ with 
$\ker \gamma_H = H$ and $a_H$ is the minimal non-negative integer 
such that $\chi(\sigma_H) = \det (\sigma_H^*)^{a_H}$. If $G$ is  
a reflection group, then $a_H=0$ or $1$, see \cite[Section~2]{Ter} for details.  

\begin{definition} \cite[Definition~4.3]{Har}\label{def rational Galois order} 
A {\em rational Galois order} is a subalgebra $\mathcal R$ in $\mathcal S(V)^G$ 
that contains $H^0(V,\mathcal O^G)$ and that is generated by a finite number of 
elements $X\in \mathcal S(V)^G$ such that, for any such $X$, there exists a 
character $\chi$ of $G$ such that $d_{\chi} X$ is holomorphic in $V$.  
\end{definition}

In \cite[Theorem 4.2]{Har} it was shown that a rational Galois order preserves 
$H^0(V,\mathcal O^G)$. In the following Lemma we prove a more general result:  a rational Galois order preserves the vector space $\mathfrak O^G$.

\begin{lemma}\label{lem rational Gal order preserve O^G}
	Let $X$ be a generator of a rational Galois order. Then $X(\mathfrak O^G) \subset \mathfrak O^G$.  
\end{lemma}

\begin{proof}
Let $\chi$ be a character of $G$ such that $d_{\chi} X$ is holomorphic in $V$. 
We take $F_{G\cdot \xi}\in \mathcal O^G_{G\cdot \xi}$ and consider a germ $P_{\eta}$ of $P= X(F_{G\cdot \xi})$ 
at a point $\eta\in V$. Denote by $d_{\eta}$ the product of all divisors $\gamma_{H}$ 
of $d_{\chi}$ such that $\gamma_{H}(\eta) =0$. The corresponding reflections $\sigma_H$ 
generate the group $G_{\eta}$.   Then $P_{\eta} = P'_{\eta}/\chi_{\eta}$, where 
$P'_{\eta}$ is a holomorphic function at $\eta$. We see that $P'_{\eta}$ is a relative 
invariant for the character $\chi_{\eta}$, where $\chi_{\eta}(h) = (h\cdot d_{\eta})/ d_{\eta}$, 
$h\in G_{\eta}$. By  \cite[Section~2]{Ter}, we have $P'_{\eta} = d_{\eta} P''_{\eta} $, 
where  $P''_{\eta}$ is holomorphic at $\eta$. Therefore, $P_{\eta}$ is also holomorphic at $\eta$. 
\end{proof}

Here is an example. 

\begin{example} 
{\rm 
Assume that we are in the setup of Subsection~\ref{sec Gelfand-Zeitlin operators}. 
Let $n\geq 4$ and consider for example the classical Gelfand-Zeitlin 
operator $E_{34}$. We will now show explicitly that $E_{34} (F)$ is holomorphic, 
where $F = \sum\limits_{g\in G}g \cdot (f_{\xi^1_3}) \in 
\bigoplus\limits_{g\in G} \mathcal O^G_{g \cdot\xi^1_3}$.  We compute, 
for example, the germ of $E_{34}(F)$ at the point $\eta:=\xi^1_3+ \xi'$, 
where $\xi'=(\delta_{32}^{ki})$.   We have  
\begin{align*}
E_{34}(F)_{\eta} =  \frac{\prod\limits_{j= 1}^{4} (v_{31}- v_{4j})}
{ (v_{31}- v_{32})(v_{31}- v_{33})} f_{\xi^1_3} \circ (\xi')^{-1} +
\frac{\prod\limits_{j= 1}^{4} (v_{32}- v_{4j})}{ (v_{32}- v_{31})
(v_{32}- v_{33})} f_{\xi'} \circ (\xi^1_3)^{-1} =\\
\frac{(v_{31}- v_{33})\prod\limits_{j= 1}^{4} (v_{32}- v_{4j})f_{\xi'} 
\circ (\xi^1_3)^{-1} - (v_{32}- v_{33})\prod\limits_{j= 1}^{4} (v_{31}- v_{4j}) 
f_{\xi^1_3} \circ (\xi')^{-1}}{(v_{32}- v_{33}) (v_{32}- v_{31})(v_{31}- v_{33})} . 
\end{align*}
We see that the polynomial in the numerator changes the sign, if we 
permute $v_{32}$ and $v_{31}$. Therefore the factor $v_{32}- v_{31}$ 
cancels and the fraction is a holomorphic function at $\eta$. Another 
important observation here is that we have to consider the holomorphic 
category instead of the algebraic one. Indeed, the rational operator $E_{34}$ 
sends a polynomial germ $F$ to the holomorphic germ $E_{34}(F)_{\eta}$ 
plus other holomorphic summands. 
}
\end{example}

Representation theory of rational Galois orders  was developed in \cite{FuZ}. 
In this paper, we generalize some of the constructions  from \cite{FuZ} for any finite 
group, see Section~\ref{sec Applications of Theorem}.

\subsection{Bases in some modules over rational Galois orders} \label{s4.3} 

Assume that there is  a $G$-invariant  neighborhood $U$ of $e\in V$ such 
that $G$ acts as a reflection group in $U$. In this case, we will call 
$G$ a {\it local reflection group}. An example of this situation is $G=S_n$ 
and $V\simeq \mathbb C^n$, where $S_n$
acts via its permutation representation. 
Another example 
is $G=S_n$ and $V= \mathbb C^n/\mathbb Z^n$. More generally, 
$G$ is a generalized Weyl group acting on $\mathbb C^n$ and
$V= \mathbb C^n/\gimel'$, where $\gimel'$ is a $G$-invariant discrete 
lattice in $\mathbb C^n$.  

In this subsection we will describe the finite dimensional vector spaces 
$\mathbb E^*_{\bar\xi}$ using divided difference operators. 
If $G$ is a local reflection group, by Chevalley-Shephard-Todd Theorem, 
the factor space $\mathcal O_e/ \mathcal J_e$ is finite dimensional and 
has dimension $|G|$. Denote by $\Delta(\Psi)$ the 
product of all $\alpha_{x}$, where $x\in \Phi^+$. For any $g\in G$, we put
$\mathcal P_g:= \partial_{g^{-1} w_0} \Delta(\Psi)$. The obtained 
polynomials are called {\it Schubert polynomials} and their images in 
$\mathcal O_e/ \mathcal J_e$ form there a basis. Note 
that $\mathcal P_{w}(e)=0$ if $w\ne e$ and $\mathcal P_{e}$ is a 
non-zero constant.  Now we can easily construct the dual basis. Consider
\begin{equation}\label{eq basis of O/J}
B(\Theta) := \langle ev_e \circ \partial_{w} \,|\,\,  w\in G\rangle,
\end{equation}
where $ev_e$ is the evaluation at $e\in V$. 
To show that $B(\Theta)$ is a basis of $(\mathcal O_e/ \mathcal J_e)^*$, we note that 
$ev_e \circ \partial_{w} (\mathcal P_g)$ is $0$, if and only if $g\ne w$.  
If $\Theta'$ is another system of simple reflections in $G$ and $\rho(\Theta) =\Theta'$, 
then  
$$
B(\Theta')  =\langle ev_e \circ \rho \circ\partial_{w} \circ \rho^{-1}\,|\,\,  w\in G \rangle 
$$ 
is another basis of $(\mathcal O_e/ \mathcal J_e)^*$. We note also that a basis 
of 
$$
(\mathcal O^{G_{\xi}}_e /\mathcal O^{G_{\xi}}_e \cap \mathcal J_e)^* 
\subset  (\mathcal O_e/ \mathcal J_e)^*
$$
is given by $\langle ev_e \circ \partial_{w}\, | \, w\in (G/G_{\xi})^{short}\rangle$, 
where $(G/G_{\xi})^{short}$ denotes the set of shortest coset representatives.

Assume that $\Theta$ is fixed. In any class $\bar{\xi}\in V/G$, we can choose a 
representative $\tilde\xi$ such that $G_{\tilde\xi}$ is parabolic with respect 
to $\Theta$. A description of the basis in $(\mathbb E^G_{\bar{\xi}})^*$
corresponding to  $B(\Theta)$ is given in the following straightforward statement:

\begin{lemma}\label{lemma basis roots}
Let $\Theta$ be a system of simple roots, $\tilde\xi$ be as above and 
$B(\Theta)$ be the corresponding basis of $(\mathcal O_e/ \mathcal J_e)^*$. 
Then $\{ ev_e \circ \partial_{w} \circ \phi_{\tilde\xi}, \,\, w\in (G/G_{\tilde\xi})^{short} \}$
is a basis of $(\mathbb E^G_{\bar{\xi}})^*$. 
\end{lemma}

We summarize the above results in the following theorem. 

\begin{theorem}\label{theor basis, reflection group} 
Let $G$ be a local  reflection group, $\Theta$ be a system of simple reflections and 
$\mathcal A$ be a subalgebra in the skew-ring $\mathcal S$ that preserves the vector space 
$\mathfrak O^{G}|_{(G \ltimes \gimel) \cdot v}$, for a subgroup $\gimel\subset V$. 
Then   
$$
\bigcup_{\tilde \xi\in \gimel/G}\{ ev_e \circ \partial_{w} \circ \phi_{\tilde\xi}, 
\,\, w\in (G/G_{\tilde\xi})^{short} \},
$$ 
is basis of the $\mathcal A$-module $M^*(G,(G \ltimes \gimel) \cdot v)$.
\end{theorem}

\begin{proof}
The statement follows from Corollary~\ref{cor M is a A-module} and Lemmata~\ref{lem action on M^*} 
and \ref{lemma basis roots}. 
\end{proof}

For instance, we have Theorem \ref{theor basis, reflection group} for all rational Galois orders.

\section{Characterization of rational Galois orders}\label{sec Structure theorem for standard algebras}
\label{s5}

Let $V$ and $G$ be as in Subsections~\ref{sect def of OGZA} and \ref{sec Standard algebras of type A}. 
Denote by $(x_{ki})$ the standard dual basis in $V^*$, that is, $x_{ki} (v) = v_{ki}$, 
where $v=(v_{st})\in V$.

\begin{theorem}\label{theor structure of norm generated algebras} 
Let $A=\sum\limits_{s=1}^p\sum\limits_{g\in G}g\cdot(f_s\phi_{\xi^{a_s}_{i_s}})\in \mathcal S(V)^G$
and assume that $A$  preserves the vector space $\mathfrak O^G$. Then $A$ is a  generator 
of a rational Galois order $\mathcal D$ (cf. Definition~\ref{def rational Galois order}). 
\end{theorem}

\begin{proof}
\noindent{\it Step 1.} 
We start by reducing the statement to the case $p=1$.  
For this, we show that $B_s:=\sum\limits_{g\in G}g\cdot(f_s\phi_{\xi^{a_s}_{i_s}})$ 
also preserves the vector space $\mathfrak O^G$, for any $s=1,\ldots, p$.  
Denote by $S_t$ the $G$-invariant polynomial 
$$
\sum\limits_{g\in G} g\cdot x_{i_t,1} =\frac{|G|}{n_{i_t}} \sum\limits_{j=1}^{n_{i_t}} x_{i_t,j},
$$ 
where $t\in \{ 1,\ldots, p\}$. Consider the operator $S_t id \in \mathcal S(V)^G$ 
and the following composition of operators
\begin{align*}
A \circ S_t id =  S_t \sum\limits_{s=1}^p\sum\limits_{g\in G}g\cdot (f_s\phi_{\xi^{a_s}_{i_s}}) 
- \frac{a_{i_t}|G|}{n_{i_t}}\sum\limits_{g\in G}g\cdot (f_t\phi_{\xi^{a_t}_{i_t}}) 
= S_t A - \frac{a_{i_t}|G|}{n_{i_t}}B_t.
\end{align*}
The operators $A \circ S_t id$, $S_t id$ and $S_t A$ all preserve 
$\mathfrak O^G$. Hence the element $B_t$ also preserves $\mathfrak O^G$, 
in case $a_t\ne 0$. 

Consider now the case $a_t=0$. Let us rewrite the operator $A$:
$$
A= \sum\limits_{a_s\ne 0}\sum\limits_{g\in G}g\cdot(f_s\phi_{\xi^{a_s}_{i_s}}) 
+ \sum\limits_{g\in G}g\cdot(f_s\phi_{\xi^{0}_{i_s}}) = 
\sum\limits_{a_s\ne 0}\sum\limits_{g\in G}g\cdot(f_s\phi_{\xi^{a_s}_{i_s}}) + H id,
$$
where $H$ is $G$-invariant. Since $A$ and the first summand  preserve 
 $\mathfrak O^G$, we deduce that $H id$ also preserves $\mathfrak O^G$. 

Therefore to prove our theorem it is enough to show that, if 
$C:=\sum\limits_{g\in G}g \cdot (f\phi_{\xi^{a}_{i}})$ preserves 
the vector space $\mathfrak O^G$, then $C\in \mathcal D$. 

\noindent{\it Step 2.} 
Assume that $C=\sum\limits_{g\in G}g \cdot (f\phi_{\xi^{a}_{i}})$ preserves 
the vector space $\mathfrak O^G$. Let us show that every function $g \cdot f$ is 
holomorphic in any Weyl chamber. In other words, we want to show that the 
function $g \cdot f$ is holomorphic at any point $w\in V$ such that 
$w= (w_{ki})$, where $w_{ki}\ne w_{kj}$, for any $k$ and $i\ne j$. 

First of all we note that, if $a=0$, then the operator $C$ is holomorphic 
at any point $v\in V$. Indeed, in this case $C= H id$, where 
$H$ is a $G$-invariant meromorphic function. Let us take 
$\sum\limits_{h\in G} h\cdot c\in \mathcal O^G_{\bar v}$, where $c\in \mathbb C\setminus\{0\}$. Then 
$$
C(\sum\limits_{h\in G} h\cdot c) = H\sum\limits_{h\in G} h\cdot c\in \mathcal O^G_{\bar v},
$$ 
where $\bar v = G\cdot v$. 
Therefore, $cH$ is holomorphic at any $h \cdot v$. Hence $H$ is holomorphic on $V$.

Assume now that $a\ne 0$. Let us take $\sum\limits_{h\in G} h\cdot F\in \mathcal O^G_{\bar v}$, 
where $F=e\cdot F\in \mathcal O^G_{v}$. Then $C(\sum\limits_{h\in G} h\cdot F)\in \mathfrak O^G$ 
is a sum of $G$-invariant germs supported at the points from the set  
$$
T=\{h\cdot v + g\cdot \xi_i^a\,\,|\,\, g,h\in G\}.
$$ 
Let us show that, from the fact that $h\cdot v + g\cdot \xi_i^a= h'\cdot v + g'\cdot \xi_i^a$ is 
a point in a Weyl chamber, it follows that $h\cdot v =h'\cdot v$ and 
$g\cdot \xi_i^a= g'\cdot \xi_i^a$. 
 
Take $w = (w_{kj}) = h\cdot v+ g\cdot \xi_i^a\in T$, a point from a Weyl chamber.  
Assume that there is $w'= (w'_{kj})= h'\cdot v+ g'\cdot \xi_i^a\in T$ such that $w'=w$.  
First of all, from $w=w'$, it follows that $w_{kj}=w'_{kj}$, for any $k\ne i$ 
and for any $j$. Further, we have two possibilities: $v_{ij} + a= v_{ip}+a$ 
or $v_{ij} + a= v_{ip}$, for some $p$. In the first case, we have $v_{ij}= v_{ip}$. 
Using that $w$ is in a Weyl chamber, we conclude that $h= id$ or $h$ is the 
transposition that sends $v_{ij}$ to $v_{ip}$. In particular, $h\cdot v = h'\cdot v$. 
Consider the case $v_{ij} + a= v_{ip}$, where $p\ne j$. In this case we have a contradiction 
with the assumption that $w$ is in a Weyl chamber. Summing up, we have $h\cdot v = h'\cdot v$, 
and hence $g\cdot \xi_i^a= g'\cdot \xi_i^a$.

Now consider the summand 
\begin{equation}\label{eq germ at x}
\sum\limits_{h_1\in G_{v}}(hh_1\cdot F) \circ (g\cdot \xi_i^a)^{-1} 
\sum\limits_{g_1\in G_{\xi_i^a}} (gg_1\cdot f) = \alpha [(h\cdot F)\circ 
(g\cdot \xi_i^a)^{-1}] (g\cdot f) \in \mathcal O^G_{w} ,
\end{equation}
where $\alpha \in \mathbb C\setminus \{0\}$, from $C(\sum\limits_{h\in G} g\cdot F)$, supported 
at the point $w=h\cdot v+ g\cdot \xi_i^a$ from a Weyl chamber. Note that, to obtain 
\eqref{eq germ at x}, we use the fact that $G_{F}= G_v$ and $G_{f}= G_{\xi_i^a}$. 
Further, putting $F=const\ne 0$, we see that $g\cdot f$ is holomorphic at $w$. 

\noindent{\it Step 3.} 
Our goal now is to show that $C\in \mathcal D$. Take 
$w = (w_{kj}) = h\cdot v+ g\cdot \xi_i^a\in T$ such that the stabilizer 
of $w$ has order $2$. We have two possibilities: 

\begin{itemize}
\item[(1)] $v_{ks} = v_{kt}$, for some $s\ne t$, $G_{w}=\{id,\sigma \}$, 
where $\sigma$ is the transposition that swaps the point $v_{ks}$ and $v_{kt}$;
\item[(2)] $v_{ij} +a = v_{ip}$, for some $j\ne p$, $G_{w}=\{id,\tau \}$, 
where $\tau$ is the transposition that swaps the point $v_{ij} +a$ and $v_{ip}$.
\end{itemize}

In the first case, as in Step~2, we get that $h\cdot f$ is holomorphic at $w$.
Consider the second possibility. The summand from $C(\sum\limits_{h\in G} h\cdot F)$ 
supported at the point $w$ is
\begin{equation}\label{eq germ holom at x2}
\begin{split}
\sum\limits_{h_1\in G_{x}}(hh_1\cdot F) \circ (g\cdot \xi_i^a)^{-1}\sum\limits_{g_1\in G_{\xi_i^a}} (gg_1\cdot f) +\\ \tau[\sum\limits_{h_1\in G_{v}}(hh_1\cdot F) \circ (g\cdot \xi_i^a)^{-1}\sum\limits_{g_1\in G_{\xi_i^a}} (gg_1\cdot f)] \in \mathcal O^G_{w}. 
\end{split}
\end{equation}
Let $F=c\in \mathbb C\setminus\{0\}$. From \eqref{eq germ holom at x2},  
we get that $g\cdot f + \tau(g\cdot f)\in \mathcal O^G_{w}$. We put
$z_1:=x_{ij}- x_{ip}$ and $z_2:=x_{ij}+x_{ip}$. Then $(z_1,z_2,x_{kt})$, 
where $(kt)\ne (ij), (ip)$, form a new coordinate system. Moreover, $z_2$ 
and $x_{kt}$ are $\tau$-invariant and $\tau(z_1)= - z_1$. 

From Step~$2$ it follows that $g\cdot f$ is a holomorphic function in 
a neighborhood of $w$, except for points $y$ with $z_1(y)=0$. Any such 
function possesses a Hartogs-Laurent series, see \cite[Section 8]{Sha}. 
Let this series be  $g\cdot f = \sum\limits_{s=q}^{\infty}H_s z_1^s$, where $H_s$ are 
holomorphic functions in $z_2$ and all $x_{kt}$. We have 
\begin{align*}
g\cdot f + \tau(g\cdot f) = \sum\limits_{s=q}^{\infty}(1 + (-1)^s)H_s z_1^s\in \mathcal O^G_{w}.
\end{align*}
We obtain that $H_s=0$, for all $s=2r<0$. 

Further, we note that $G_{h\cdot v}=\{id \}$ or $G_{h\cdot v} = \{id, \theta \}$, where $\theta$ is an 
involution that swaps $v_{ij}$ with some $v_{ij'}$, where $j'\ne p$. 
In the first case, set $h\cdot F=z_1\in  \mathcal O^{G_{h\cdot v}}_{h\cdot v}$. 
In the second case, set $h\cdot F=z_1 + \theta (z_1) \in  \mathcal O^{G_{h\cdot v}}_{h\cdot v}$. 
In both cases, using \eqref{eq germ holom at x2}, we obtain 
\begin{align*}
z_1(\sum\limits_{s=q}^{\infty}H_s z_1^s - \sum\limits_{s=q}^{\infty}(-1)^sH_s z_1^s) \in \mathcal O^G_{w}. 
\end{align*} 
This is possible only if $H_s=0$, for $s< 1$. Therefore $g\cdot f$ has only a simple pole at $w$. 

Denote by $\Delta$ the product of all $x_{ki}-x_{kj}$, where $i\ne j$. 
Summing up, above we proved that $f$ is holomorphic in any Weyl chamber and it 
has a simple pole or it is holomorphic at all points with the stabilizer 
of order $2$.  This implies that $f\Delta $ is holomorphic  at all point 
with the stabilizer of order $1$ or $2$. By the Riemann extension theorem, 
see e.g. \cite[Corollary~6.4]{Dem}, singularities of  codimension at least $2$ 
are removable.  It follows that $f\Delta= H$ is homomorphic in $V$.  
The proof is complete.
\end{proof}

\begin{corollary}\label{coroll subalgebras in univ ring D}
Let  $\mathcal A\subset \mathcal S(V)^G$ be a finitely generated over 
$H^0(V,\mathcal O^G)$  standard algebra of type $\mathbb A$ that 
preserves the vector space $\mathfrak O^G$. Then $\mathcal A$ is a rational Galois order. 
\end{corollary}

This description of standard algebras of type $\mathbb A$ 
that preserve  $\mathfrak O^G$ is surprising.  
It would be interesting to prove an analog of this result 
(or to find a counter-example) for other reflection groups.

\section{Applications of Theorem \ref{theor action in the bundle} to Gelfand-Zeitlin modules}
\label{sec Applications of Theorem} \label{s6}

Let $\mathcal A$ be a subalgebra in $\mathcal S(V)$ that preserves 
the vector space $\mathfrak O^G|_{(G \ltimes \gimel) \cdot v}$, for some $v\in V$, 
and $\mathcal B$ be the algebra of global $G$-invariant functions on $V$. 
Then, by Corollary \ref{cor M is a A-module} and Proposition~\ref{prop all modules are GZ},  
$M(G,(G \ltimes \gimel) \cdot v)$ and $M^*(G,(G \ltimes \gimel) \cdot v)$ 
are $\mathcal A$-modules. These $\mathcal A$-modules and 
their submodules were studied,  for some special cases, simultaneously 
and independently  in \cite{RZ} (the case of 
$\mathcal A= U(\mathfrak{gl}_n (\mathbb C))$) and in \cite{EMV}
(the case of $\mathcal A$ being an orthogonal Gelfand-Zeitlin algebra).
The case when $\mathcal A$ is a rational Galois 
orders corresponding to any reflection group was later considered in \cite{FuZ}. 
In this section, we show how to obtain \cite[Section 5.6, Theorem]{RZ}, \cite[Theorem~10]{EMV} 
and \cite[Theorem~7.4]{FuZ} using Corollary~\ref{cor M is a A-module}, 
Theorem~\ref{theor homomorphism of A-modules} and Proposition~\ref{prop all modules are GZ}.

\subsection{The case of orthogonal Gelfand-Zeitlin algebras} \label{s6.1}
 
Let $V$, $\gimel$ and $G$ be as in Subsection~\ref{sect def of OGZA}.
The classical Gelfand-Zeitlin operators $E_{st}$ and the generators of 
the orthogonal Gelfand-Zeitlin algebra $E_k$ and $F_k$  are rational, 
however as it was shown in Lemma~\ref{lem rational Gal order preserve O^G}, 
we have $E_{k} (\mathfrak O^G) \subset \mathfrak O^G$ and 
$F_{k} (\mathcal O^G) \subset \mathcal O^G$. Clearly, the same 
holds for $E_{st}$. Further let us take $v'\in V$. It is easy 
to see that there exists $v\in \gimel \cdot v'$ such that $G_v$ 
includes all stabilizers $G_w$, where $w\in \gimel \cdot v'$. 

\begin{lemma}\label{lem G_v is stan of the orbit}
We have $G_v =G_{\gimel \cdot v}$. 
\end{lemma}

\begin{proof}
For $v=(v_{ki})$, the following holds: if $v_{ki}-v_{kj}\in \mathbb Z$, then 
$v_{ki} = v_{kj}$. Further, it is clear that $G_v \subset G_{\gimel \cdot v}$. 
If $g\in G_{\gimel \cdot v}$, then $g\cdot v\in \gimel \cdot v$ 
or, equivalently, $g\cdot v - v\in \gimel $. Hence 
$(g\cdot v)_{ki} - v_{ki} \in \mathbb Z$ and thus
$(g\cdot v)_{ki} = v_{ki}$, implying $g\in G_v$. 
\end{proof}

By Lemma~\ref{lem G_v is stan of the orbit} and Proposition~\ref{prop isom pr}, 
we get that $M(G_v,\gimel \cdot v)$ and $M^*(G_v,\gimel \cdot v)$ are 
$\mathcal A$-modules. From Proposition~\ref{prop all modules are GZ} 
it follows that these modules are Harish-Chandra modules and therefore
Gelfand-Zeitlin modules. This recovers the corresponding 
results from \cite{RZ} and \cite{EMV}.

\subsection{The case of rational Galois orders}\label{s6.2}

Let $V$, $\gimel$ and $G$ be as in Subsection~\ref{sec universal ring}. 
Take $v\in V$ and let $H$ be a subgroup in $G$ that contains 
all stabilizers $G_w$, where $w\in \gimel \cdot v$. Then it is easy 
to check (we refer to Theorem~\ref{theor structure} for details)  
that a rational Galois order $\mathcal A$ preserves the vector space 
$\mathfrak O^H|_{(H \ltimes \gimel) \cdot v}$. By 
Lemma~\ref{lem rational Gal order preserve O^G}, the algebra  
$\mathcal A$ preserves also the vector space $\mathfrak O^G|_{(G \ltimes \gimel) \cdot v}$. 
Therefore we may apply Theorem~\ref{theor homomorphism of A-modules} 
to obtain a family of the corresponding modules. 
In the case when $H$ is a reflection group and satisfies some other 
conditions (it has to be parabolic with respect to a fixed system of
simple roots), the $\mathcal A$-module $Im (\Upsilon^*)$, cf
Theorem~\ref{theor homomorphism of A-modules}, was 
constructed in \cite[Theorem~7.4]{FuZ}. This recovers the
corresponding result of \cite{FuZ}.

\section{Structure theorem for rational Galois order}\label{s7}

\subsection{Further examples of algebras that preserve the vector space $\mathfrak O^G$}\label{section generalization of OGZ}\label{s7.1}

In this section we assume that  $G$ is a reflection group on 
$V\simeq\mathbb C^n$. Let us fix a system $\Psi$ of simple roots and 
let $\Theta$ be the set of the corresponding simple reflections. 
Our goal now is to define two classes of algebras preserving the 
vector space $\mathfrak O^G$. As above, we denote by $\circ$ composition 
of operators or the product in $G \ltimes V$ and we use $\cdot $ 
to denote the action of $G$. For example, if $g\in G$ and $\xi\in V$, 
then $g \cdot \xi = g \circ \xi \circ g^{-1}$ and 
$g \cdot \phi_{\xi} = g \circ \phi_{\xi} \circ g^{-1}$.

{\it Algebras of type $I$.} These are subalgebras 
of $\mathcal S$ generated by  elements of the form 
$\sum\limits_i\partial_{w_i} \circ p_i\phi_{v_i}$, 
where, for each $i$, the stabilizer  $G_{v_i}$ of $v_i\in V$ 
is parabolic with respect to $\Theta$, the function $p_i$ is 
$G_{v_i}$-invariant and holomorphic (or meromorphic, or 
rational or polynomial) and $w_i$ is the longest element in
$(G/G_{v_i})^{short}$.

{\it Algebras of type $II$.} These are subalgebras of 
$\mathcal S$ generated by  elements in the form 
$\sum\limits_i\partial_{w_i} \cdot p_i\phi_{v_i}$, where $v_i$,  
$p_i$ and $w_i$ are as in type I (note the difference of using $\cdot$
in type $II$ instead of $\circ$ in type $I$).

Let $\mathbf A$ be an algebra of type $I$.  
Denote by $\gimel$ the subgroup of $V$ generated by all 
possible $g\cdot v_i$, where $g\in G$ and $v_i$ appears in 
a generator of $\mathbf{A}$, see above.  

\begin{proposition}\label{prop invariance}
Let $E=\sum\limits_i \partial_{w_i} \circ p_i\phi_{v_i}$ 
be a generator of the algebra $\mathbf A$. If all $p_i$ are holomorphic in $V$, then 
$$
E (\mathfrak O^G) \subset \mathfrak O^G.
$$
\end{proposition}

\begin{proof}
Take a simple reflection $\tau\in\Theta$. Then
\begin{equation}\label{eqnn124}
(id -\tau) \circ \partial_{w_i} \circ p_i\phi_{v_i} = \gamma_{\tau} \partial_{\tau} \circ\partial_{w_i} \circ p_i\phi_{v_i}.  
\end{equation}
Since $w_i$ is the longest element in $(G/G_{v_i})^{short}$, the operator
$\partial_{\tau}\circ\partial_{w_i}$ is either zero or can be written
as $\partial_{u}\circ \partial_{s}$, where $\partial_s\in G_{v_i}$.
Therefore the right hand side of \eqref{eqnn124} 
is identically zero on $\mathfrak{M}^G$. 
Hence, for any  $F\in \mathfrak{M}^G$ and $g\in G$, we have 
$g\circ \partial_{w_i} \circ p_i\phi_{v_i} (F)= \partial_{w_i} 
\circ p_i\phi_{v_i}(F)$ implying 
$\partial_{w_i} \circ p_i\phi_{v_i}(\mathfrak M^G ) \subset \mathfrak M^G$.  

Further, we have $\partial_{\tau}(\mathfrak O)\subset \mathfrak O$. Indeed, let us take $f_x\in \mathcal O_x$ and consider $\partial_{\tau}(f_x)$. If $\tau(x)=x$, then $\gamma_{\tau}(x)=0$. In this case $\gamma_{\tau}$ is a divisor of $f_x - \tau(f_x)\in \mathcal O_x$. Therefore, $\partial_{\tau}(f_x)\in \mathcal O_x$. If $\tau(x)\ne x$, then $\gamma_{\tau}(x)\ne 0$. Hence $f_x/\gamma_{\tau} \in \mathcal O_x$ and $\tau(f_x)/\gamma_{\tau} \in \mathcal O_{\tau(x)}$.
\end{proof}

Let $\mathbf A$ be an algebra of type $I$ and $\mathbf B$
be an algebra of type $II$. Assume that for each generator $E=\sum\limits_i \partial_{w_i} \circ p_i\phi_{v_i}$ of $\mathbf A$ there
is a generator $E'=\sum\limits_i \partial_{w_i} \cdot p_i\phi_{v_i}$ of $\mathbf B$ and vice versa.
 The next lemma describes 
when the actions of $E$ and $E'$ coincide.

\begin{lemma}\label{lem A=B} Assume that all $p_i$ are holomorphic. 
We have the equality of operators 
$$
E|_{\mathfrak O^G} = E'|_{\mathfrak O^G}.
$$
Therefore, the actions of algebras $\mathbf A$ and $\mathbf B$ as above on 
$\mathfrak O^G$ coincide. 
\end{lemma}

\begin{proof} 
Consider first the operators  $\partial_{\rho}\circ f\phi_{x} $ and
$\partial_{\rho}\cdot f\phi_{x}$, where $f\in \mathcal M$ is any 
meromorphic function and $\rho \in G$ is any (not necessary longest) 
element with reduced expression $\rho=\tau_1\tau_2\cdots\tau_k$.
Let us prove, by induction on $k$, 
that  
$$\partial_{\rho}\circ f\phi_{x}|_{\mathfrak M^G} = 
\partial_{\rho}\cdot f\phi_{x}|_{\mathfrak M^G} .$$ 
For $k=1$, the claim is obvious. To establish the induction step, we have
\begin{align*}
\partial_{\tau_1} \circ \cdots \circ \partial_{\tau_k}\circ f\phi_{x}|_{\mathfrak M^G}= \partial_{\tau_1} \circ \cdots \circ \partial_{\tau_{k-1}}\circ (\partial_{\tau_k} \cdot f\phi_{x})|_{\mathfrak M^G} = \\
\partial_{\tau_1} \circ \cdots \circ \partial_{\tau_{k-1}}\circ (f/\gamma_{\tau_k}\phi_{x} - (\tau_k\cdot f)/\gamma_{\tau_k} \tau_k \cdot \phi_{x}) |_{\mathfrak M^G} = \\
\partial_{\tau_1} \circ \cdots \circ \partial_{\tau_{k-1}}\cdot (f/\gamma_{\tau_k}\phi_{x} - (\tau_k\cdot f)/\gamma_{\tau_k} \tau_k\cdot \phi_{x}) |_{\mathfrak M^G} = \\
\partial_{\tau_1}  \cdots  \partial_{\tau_k}\cdot f\phi_{x}|_{\mathfrak M^G}.
\end{align*}
The result now follows from Proposition \ref{prop invariance}. 
\end{proof}

\subsection{Structure theorem for rational Galois order}
\label{sec Structure theorem for rational Galois order}\label{s7.2}

In this section we assume that  $G$ is a reflection group on 
$V\simeq\mathbb C^n$, where $\Phi$ is a root system
with basis $\Psi$  and $\Theta$ is the set of corresponding 
simple reflections (cf. Subsection~\ref{sec reflection group, roots, div diff}). 
We have the decomposition $\Phi = \Phi^+\cup \Phi^-$ corresponding to $\Psi$. 
Consider the following product of linear functions on $V$:
\begin{align*}
\Delta := \prod_{x \in \Phi^+} \gamma_x,
\end{align*}
where $\gamma_x (v) = (x,v)$, for any $v\in V$, see 
Section~\ref{sec reflection group, roots, div diff}. 
We have $\sigma_{x} \cdot \Delta  = -\Delta$, for any simple 
reflection $\sigma_{x}\in \Theta$. If $G=S_n$, then $\Delta$ 
may be identified with the classical Vandermonde determinant. 

Let us take $v\in V$ such that the stabilizer $G_v$ is 
parabolic in $G$ with respect to $\Theta$. Denote by $\Delta'$ 
the product of $\gamma_x$, where  $\sigma_x$ is a reflection 
in $G_v$. Let us take also a polynomial (or a holomorphic function) 
$p'$ and let $w$ be the longest element in $(G/G_{v})^{short}$.

Consider an element of the form $\sum\limits_{\tau\in G}   
\tau \cdot (\frac{p'}{\Delta} \phi_{v})$ from a rational Galois 
order $\mathcal A$, see Subsection~\ref{sec universal ring}. 
We always can choose  $p'$ such that it satisfies
$$
\tau \cdot p' = \chi(\tau) p',\quad \text{where} \quad \chi(\tau): = \frac{\tau \cdot \Delta}{\Delta}, \quad \text{ for }\quad\tau\in G_v. 
$$ 
Therefore we have $p'= \Delta' p$, where $p$ is a $G_v$-invariant polynomial 
or a holomorphic function (cf. Subsection~\ref{sec universal ring}). 
In other words, if $\Phi'\subset \Phi$ is the 
root subsystem corresponding to $G_v$, then 
\begin{align*}
\Delta' := \prod_{x \in \Phi'^+} \gamma_x,
\end{align*}
where $\Phi'^+$ is the subsystem of positive roots generated by $\Psi \cap \Phi'$. If $w_0$ is the longest element in $G$, then we have the following equality on global rational functions
\begin{equation}\label{eq d_w_01}
\partial_{w_0} = \sum_{\tau\in G} \tau \cdot \frac{1}{\Delta},
\end{equation}
see \cite[Section IV, Proposition 1.6]{Hill}. From Dedekind's Theorem it follows that the operators (\ref{eq d_w_01}) are equal as elements in $\mathcal S$. Therefore we have
\begin{equation}\label{eq d_w_0}
\partial_{w_0}|_{\mathfrak O} = \sum_{\tau\in G} \tau \cdot \frac{1}{\Delta}|_{\mathfrak O},
\end{equation}
 The following theorem generalizes \cite[Proposition 7]{EMV}.

\begin{theorem}[Structure Theorem]\label{theor structure}
{\hspace{1mm}}

\begin{enumerate}[$($a$)$]
\item\label{claim1} We have
\begin{equation}\label{eq structure of Galois algebra}
\sum_{\tau\in G}   
\tau \cdot \frac{\Delta'}{\Delta} p \phi_{v}\big|_{\mathfrak O^G} 
= a \partial_w \circ p\phi_v\big|_{\mathfrak O^G},
\end{equation}
where $a\ne 0$ is a scalar. 
\item\label{claim2} 
Let $G'$ be any subgroup in $G$ which is parabolic with respect
to  $\Theta$. Then 
\begin{equation}
\sum_{\tau\in G}   
\tau \cdot \frac{\Delta'}{\Delta} p \phi_{v}\big|_{\mathfrak O^G}  =\sum_{s=1}^k \partial_{w_s} \circ t_s \phi_{v_s}|_{\mathfrak O^G},
\end{equation}
where $w_s\in G'/G'_{v_s}$ is the longest reduced element and $t_s$ 
are rational functions defined in Weyl chambers and at 
$\ker \gamma_x$, where $x\in \Phi^+$ and  $\sigma_x\in G'$. 
\end{enumerate}
\end{theorem}

\begin{proof}  
Note that we always can choose $v$ such that $G_v$ is parabolic with respect to $\Theta$.  Using (\ref{eq d_w_0}), we have
\begin{align*}
\sum_{\tau\in G}   
\tau \cdot \frac{\Delta'}{\Delta} p \phi_{v}\big|_{\mathfrak O^G} = \sum_{\tau\in G} \tau \cdot \frac{1}{\Delta} \Delta' p \phi_{v}|_{\mathfrak O^G} = \partial_{w_0} \circ \Delta' p \phi_{v}|_{\mathfrak O^G} =\\
 \partial_{w} \circ \partial_{w'_0} \circ \Delta' p \phi_{v}|_{\mathfrak O^G}= \partial_{w} \circ  p \phi_{v} \partial_{w'_0} \circ (\Delta')|_{\mathfrak O^G}= a \partial_{w} \circ  p \phi_{v} |_{\mathfrak O^G},
\end{align*}	
where $w'_0$ is the longest element in $G_v$. This implies claim \eqref{claim1}.

To prove claim~\eqref{claim2}, let $G'$ be a subgroup in $G$ which 
is parabolic with respect to $\Theta$. We have
\begin{align*}
\sum_{\tau\in G}   
\tau \cdot \frac{\Delta'}{\Delta} p \phi_{v} =  \sum_{\tau\in G'}  \tau 
\cdot  \sum_{s=1}^{k}
\tau_s \cdot \frac{\Delta'}{\Delta} p \phi_{v}.
\end{align*}
Here $\tau_s \in G'\backslash G$ is a  coset representative and 
$k= |G'\backslash G|$. Note that we can choose the representatives $\tau_s$ 
such that $\phi_{v_s} :=\tau_s \cdot \phi_{v}$ has a parabolic stabilizer 
$G'_{v_s}$ with respect to $\Theta$. 

Denote by $\tilde{\Delta}$  the product of $\gamma_x$, 
where $x\in \Phi^+$ and  $\sigma_x\in G'$, and by $\tilde{\Delta}_s$  
the product of $\gamma_x$, where $x\in \Phi^+$ and  
$\sigma_x\in G'_{v_s}$. Clearly, $\tilde{\Delta}$ is a divisor 
of $\Delta$ and $\tilde{\Delta}_s$ is a divisor of 
$\tau_s\cdot  \Delta'$. Denote by $l(\tau_s)$ the length of $\tau_s$.  We have
\begin{align*}
\tau_s \cdot \frac{\Delta'}{\Delta} p \phi_{v} =  \frac{(-1)^{l(\tau_s)}\tau_s\cdot \Delta' }{\Delta} p_s \phi_{v_s} =  \frac{\tilde{\Delta}_s}{\tilde{\Delta}} \frac{(-1)^{l(\tau_s)} (\tau_s\cdot \Delta')/ \tilde{\Delta}_s }{\Delta/ \tilde{\Delta}} p_s \phi_{v_s},
\end{align*}
where $p_s:= \tau_s \cdot p$.  We put 
$$
t_s:= \frac{(-1)^{l(\tau_s)} (\tau_s\cdot \Delta')/ \tilde{\Delta}_s }{\Delta/ \tilde{\Delta}} p_s.
$$
We see that $t_s$ are rational functions defined in Weyl 
chambers and at  $\ker \gamma_x$, where $x\in \Phi^+$ and  $\sigma_x\in G'$. 

Using \eqref{claim1}, we obtain
\begin{align*}
\sum_{s=1}^{k} \sum_{\tau\in G'} \tau \cdot  
\frac{\tilde{\Delta}_s}{\tilde{\Delta}} t_s \phi_{v_s} |_{\mathfrak O^G}= 
\sum_{s=1}^{k} a_s \partial_{w_s} \circ t_s\phi_{v_s}|_{\mathfrak O^G},
\end{align*}
where $a_s\ne 0$ are scalars and $w_s\in G'/ G'_{v_s}$
are longest element in the set of shortest coset representatives. 
\end{proof}

In the case $G=S_n$, formula~\eqref{eq structure of Galois algebra} 
was conjectured by the second author in \cite{ViFAP} and later
independently proved in \cite{RZ,EMV}.  It was extended
to an arbitrary reflection group in \cite{FuZ} where it was also shown
that it plays a crucial role  in construction and study of
simple Gelfand-Zeitlin modules for rational Galois orders.  

Consider a rational Galois order $\mathcal A$ as above. 
Fix $v\in V$ and denote by $H$ the subgroup in $G$ generated 
by all stabilizers $G_v$, where $v\in \gimel \cdot v$. 
In the proof of Theorem~\ref{theor structure} we 
obtained the following expression
\begin{align*}
A=\sum_{\tau\in G}   
\tau \cdot \frac{\Delta'}{\Delta} p \phi_{v} = \sum_{s=1}^{k} \sum_{\tau\in H} \tau \cdot  \frac{\tilde{\Delta}_s}{\tilde{\Delta}} t_s \phi_{v_s},  
\end{align*}
where $\tilde{\Delta}$ is the product of $\gamma_x$, 
for $x\in \Phi^+$ and  $\sigma_x\in H$, and $\tilde{\Delta}_s$  
the product of $\gamma_x$, for $x\in \Phi^+$ and  $\sigma_x\in H_{v_s}$. 
By Lemma~\ref{lem rational Gal order preserve O^G}, the operator $A$ 
preserves the vector spaces $\mathfrak O^G$ and $\mathfrak O^H$. 
Therefore we have the families of modules given by 
Theorem~\ref{theor homomorphism of A-modules}. In particular,
we have the $\mathcal A$-modules 
$M^*(G,(G \ltimes \gimel) \cdot v)$ and $M^*(H, \gimel \cdot v)$. 
A basis of these modules is constructed in 
Theorem~\ref{theor basis, reflection group}. 
Using Theorem~\ref{theor structure}, we get the following fairly explicit
result.

\begin{corollary}\label{theor main A-modules explicite action of algebra A}
With respect to the basis of Theorem~\ref{theor basis, reflection group},
the action of $\mathcal A$ on the modules $M^*(G,(G \ltimes \gimel) \cdot v)$ or 
$M^*(H, \gimel \cdot v)$ can be computed using the following formula: 
\begin{equation}
\begin{split}
(ev_0 \circ \partial_{w}\circ \phi_{\xi}) \circ A|_{\mathfrak O^{G}}  = 	(ev_0 \circ \partial_{w}\circ \phi_{\xi}) \circ (\partial_w \circ p\phi_v)|_{\mathfrak O^{G}} =\\ \sum_{s=1}^{k} a_s ev_0 \circ \partial_{w}  \circ \partial_{w_s} \circ (\phi_{\xi} \cdot t_s)  \phi_{\xi\circ v_s}|_{\mathfrak O^{G}},
\end{split}
\end{equation}
where $a_s\in \mathbb C\setminus\{0\}$, cf. Theorem~\ref{theor structure}. 
Here $t_s$ and $v_s$ correspond to $G'= G_{\xi}$. 
\end{corollary}

\section{A construction of simple modules and sufficient conditions for simplicity }\label{s8}

\subsection{Canonical simple Harish-Chandra modules}\label{sec canonical simple}

In this section we constract a family of simple modules which we will call {\it canonical Harish-Chandra modules}. This construction generalizes the corresponding constructions from \cite{EMV} and \cite{Har}. Assume that $V$ is a complex-analytic Lie group, $G$ is a finite group, $\gimel\subset V$ 
is a subgroup and $v\in V$.   Let $\mathcal A\subset S(V)^G$ be a subalgebra 
containing $H^0(V, \mathcal O^G)$, which 
preserves the vector space $\mathfrak O^G|_{(G\ltimes \gimel) \cdot v}$. Consider the $\mathcal A$-module $M(G, (G\ltimes \gimel) \cdot v)$.  Denote by $N_{\bar w}$, where $\bar w= G\cdot w$ for some $w\in (G\ltimes \gimel) \cdot v$, the submodule in  $M(G, (G\ltimes \gimel) \cdot v)$ generated by $\tilde 1_{\bar w}\in \mathbb E^G_{\bar w}$, where $\tilde 1_{\bar w}$ is the class generated by the constant function $1$.

\begin{proposition}
	Assume that $H^0(V, \mathcal O^G)$ separates orbits in $(G\ltimes \gimel) \cdot v$ and that the $H^0(V, \mathcal O^G)$-module $\mathbb E^G_{\bar w}$ is generated by $\tilde 1_{\bar w}$. Then $N_{\bar w}$ has a unique maximal submodule. 
\end{proposition}

\begin{proof}
	The  unique maximal submodule is the sum of all submodules $N'$ in $N_{\bar w}$ such that  $N'\cap \mathbb E^G_{\bar w} \subset \mathfrak n_{\bar w} \cdot \tilde 1_{\bar w}$, where $\mathfrak n_{\bar w}\subset H^0(V, \mathcal O^G)$ is the ideal of all $G$-invariant functions that  are equal to $0$ at $\bar w$.
\end{proof}

The quotient of $N_{\bar w}$ by its unique maximal submodule is denoted $L_{\bar w}$ and is
called the {\it canonical simple Harish-Chandra module associated to $\bar w$}.

\subsection{Standard algebras of type $\mathbb A$} \label{s8.1}

Let $V$ and $G$ be as in Subsection~\ref{sect def of OGZA} and 
\ref{sec Standard algebras of type A}.
As we have seen in Corollary~\ref{coroll subalgebras in univ ring D},  
a finitely generated over $H^0(V,\mathcal O^G)$ standard algebra of
type $\mathbb A$ that preserves the vector space $\mathfrak O^G$ is a rational 
Galois order. Consider a special case of such algebras, the algebra 
$\mathcal A$ that is generated by $H^0(V,\mathcal O^G)$ and by 
the following elements:
\begin{align*}
E_i=\sum\limits_{g\in G}g\cdot(\frac{\Delta' H^E_i}{\Delta}\phi_{\xi^{a}_{i}}), \quad F_i=\sum\limits_{g\in G}g\cdot(\frac{\Delta' H^F_i}{\Delta}\phi_{\xi^{-a}_{i}}), \quad i=1,\ldots, n,
\end{align*} 
where $a\in \mathbb C\setminus \{ 0\}$, $H^E_i, H^F_i$ are holomorphic 
functions in $V$ such that we have $G_{H^E_i} = G_{H^F_i} = G_{\xi^{a}_{i}}$, 
for  $i=1,\ldots, n$, and $\Delta$ and $\Delta'$ are as in 
Subsection~\ref{sec Structure theorem for rational Galois order}. 

Let $\gimel$ be a subgroup in $V$ generated by $\xi^{a}_{i}$, 
where $i=1,\ldots, n$, and $v'\in V$ be any point. In this case, 
for $G_{\gimel\cdot v'}$ we have an analogue of 
Lemma~\ref{lem G_v is stan of the orbit}. That is, there exists 
$v\in \gimel\cdot v'$ such that $G_{\gimel\cdot v'} = G_v$. 
The module $M^*(G_{\gimel\cdot v}, \gimel\cdot v) = M^*(G_{v}, \gimel\cdot v)$
was studied in \cite[Theorem~11]{EMV}. 
More precisely, in \cite{EMV} the following theorem was proved. 

\begin{theorem}\cite[Theorem~11]{EMV}
Assume that $H^E_i, H^F_i$, where $i=1,\ldots, n$, 
have no zeros on $\gimel\cdot v$. Then the 
$\mathcal A$-module $M^*(G_{v}, \gimel\cdot v)$ is irreducible. 
\end{theorem}

In \cite{EMV}, this theorem was proved only for a special choice 
of functions $H^E_i, H^F_i$. However exactly the same proof as in 
\cite{EMV} works for any functions $H^E_i, H^F_i$. This fact was 
noticed in \cite[Theorem~8.5]{FuZ}, where the result \cite[Theorem~11]{EMV} 
was discussed in detail. 

\subsection{Regular modules}\label{8.2}

Assume that $V$ is a complex-analytic Lie group, $\gimel\subset V$ 
is a subgroup and $v\in V$.   Let $\mathcal A\subset S(V)$ be a 
finitely generated, over $H^0(V, \mathcal O)$, subalgebra which 
preserves the vector space 
$\mathfrak O|_{\gimel \cdot v}= \mathfrak O^e|_{\gimel \cdot v}$. 
We denote by $\Gamma$ an oriented graph that is defined in the 
following way. The vertices of $\Gamma$ are all points from 
$\gimel \cdot v$ and  we connect $x$ and $y$ with an arrow 
$x \to y$ if there exists $A=\sum f_i\phi_{\xi_i}\in \mathcal A$  
and $i_0$ such that $\phi_{\xi_{i_0}} (x) = y$ and $f_{i_0}(y)\ne 0$. 
Note that, in this case, all $f_i$ are holomorphic in $\gimel \cdot v$.  

\begin{proposition}\label{propn231}
Assume that  $\mathcal A\subset S(V)$ is a finitely generated 
over $H^0(V, \mathcal O)$ subalgebra that preserves the vector space 
$\mathfrak O|_{\gimel \cdot v}= \mathfrak O^e|_{\gimel \cdot v}$, 
$H^0(V, \mathcal O)$ separates points of $\gimel \cdot v$ and 
$\Gamma$ is connected as an oriented graph. Then the 
$\mathcal A$-module $M(\{e\}, \gimel \cdot v )$ is irreducible. 
\end{proposition}
    
\begin{proof}
First of all, we note that the $\mathcal A$-module 
$M(\{e\}, \gimel \cdot v )$ is a direct sum of 
$\mathbb E_{\xi} =  \phi_{\xi} 
(\mathcal O^{\{e\}}_e /(\mathcal O^{\{e\}}_e \cap \mathcal J_e)) \simeq \mathbb C$.
In other words, $M(\{e\}, \gimel \cdot v )$ is a vector space of all finite linear
combinations of points $v_s\in \gimel \cdot v$.  
 
Let $\sum a_s v_s$ be an element in a submodule $N$. Since $H^0(V, \mathcal O)$ 
separates points of $\gimel \cdot v$, we see that 
$v_s\in N$ for any $s$. Let us take a submodule 
$N'\subset M(\{e\}, \gimel \cdot v )$ that contains a point
$x\in \gimel \cdot v$. Further let $y\in \gimel \cdot v$. 
Since $\Gamma$ is connected as an oriented graph, there exists a sequence 
$$
x_0=x, x_1,\ldots, x_{n-1}, x_n=y
$$ 
such that the path $x_0 \to x_1\to \cdots \to  x_n $ connects $x$ and $y$. 
Assume, by induction, 
that we proved that $x_{s-1}\in N$. From our assumptions, there exists  
$A=\sum f_i\phi_{\xi_i}\in \mathcal A$  and $i_0$ such that 
$\phi_{\xi_{i_0}} (x_{s-1}) =x_s$ and $f_{i_0}(x_s)\ne 0$. 
We have $A(x_{s-1}) = \sum f_i(x_s)\phi_{\xi_i}(x_{s-1})\in N$. 
Therefore $x_s\in N'$. 
\end{proof}

Assume that $\mathcal A$ is generated by  $H^0(\gimel \cdot v, \mathcal O)$ 
and, additionally, by  a finite set of elements $E_i = \sum f_{ij}\phi_{\xi_{ij}}$. 
Let $\gimel$ be the group generated by all $\xi_{ij}$. 
Denote by $Q(\xi_{ij})$  the monoid generated by all $\xi_{ij}$.

\begin{proposition}
Assume that  
\begin{enumerate}[$($i$)$]
\item\label{assum1} $H^0(V, \mathcal O)$ separates points of $\gimel \cdot v$;
\item\label{assum2} $Q (\xi_{ij}) = \gimel$;
\item\label{assum3} every $f_{ij}$ has no zeros at $\gimel \cdot v$.
\end{enumerate}
Then the $\mathcal A$-module $M(\{e\}, \gimel \cdot v )$ is irreducible. 
\end{proposition}

\begin{proof}
Due to assumptions \eqref{assum1} and \eqref{assum3}, to be able to use
Proposition~\ref{propn231}, we only need to show that
$\Gamma$ is connected. The latter, however, follows directly from 
assumption~\eqref{assum2}.  Therefore the claim follows from
Proposition~\ref{propn231}
\end{proof}

\subsection{Singular modules}\label{s8.3}

Assume that $V$ is a complex-analytic Lie group, $\gimel\subset V$ 
is a subgroup and $v\in V$.  Let $\mathcal A\subset S(V)^{G_{\gimel\cdot v}}$ be a 
finitely generated over $H^0(V, \mathcal O^{G_{\gimel\cdot v}})$ subalgebra which 
preserves the vector space 
$\mathfrak O^{G_{\gimel\cdot v}}|_{\gimel \cdot v}$. 
Assume that $H^0(V, \mathcal O^{G_{\gimel\cdot v}})$ separates $G_{\gimel\cdot v}$-orbits 
in $\gimel\cdot v$ and that the $H^0(V, \mathcal O^{G_{\gimel\cdot v}})$-module 
$\mathbb E^{G_{\gimel\cdot v}}_{\bar\xi}$, see \eqref{eq vector space sum E_xi_i}, 
is generated by a non-trivial constant $c\in \mathbb C\setminus\{ 0\}$,
for any $\xi\in \gimel\cdot v$.

We denote by $\Gamma$ the oriented graph defined as follows:
\begin{itemize}
\item the vertices of $\Gamma$ are all $G_{\gimel\cdot v}$-orbits 
in $\gimel \cdot v$;
\item for two orbits  $\bar \xi$ to $\bar\eta$, there is
an oriented  arrow from $\bar \xi$ to $\bar\eta$, 
if there exists $A=\sum f_i\phi_{\xi_i}\in \mathcal A$ such that the 
function $H:=\sum\limits_{(g,h,i)\in\Lambda} h\cdot f_i$, cf. \eqref{eq A(F)_eta} for $X=1$, exists and is not equal to $0$ at $\eta$. (Note that the function $H$ depends on the orbits $\bar \xi$ and $\bar \eta$ and on the element $A$.)
\end{itemize}

\begin{theorem}\label{thmnn54} 
In the above situation, we have:

\begin{enumerate}[$($i$)$]
\item\label{thmnn54.1} For every $\bar\xi$, the module $M(G_{\gimel\cdot v}, \gimel \cdot v )$
has a unique submodule $N(\bar\xi)$ which is maximal, with respect to inclusions,
among all submodules of $M(G_{\gimel\cdot v}, \gimel \cdot v )$ that do not contain
$\mathbb E^{G_{\gimel\cdot v}}_{\bar\xi}$.
\item\label{thmnn54.2} 
If $\Gamma$ is connected as an oriented graph, then
$M(G_{\gimel\cdot v}, \gimel \cdot v )$ is generated by the class of a non-trivial
constant function and also has a unique maximal submodule.
\end{enumerate}
\end{theorem}

\begin{proof} 
Claim~\eqref{thmnn54.1} follows from Proposition \ref{propn231}.
Further we have
$$
M(G_{\gimel\cdot v}, \gimel \cdot v )= 
\bigoplus_{\bar\xi \in \gimel \cdot v/G_{\gimel\cdot v}} \mathbb E^{G_{\gimel\cdot v}}_{\bar\xi},
$$ 
see (\ref{eq vector space sum E_xi_i}). Denote by $N$ the $\mathcal A$-submodule 
generated by the class $c_{\bar \xi}\in \mathbb E^{G_{\gimel\cdot v}}_{\bar\xi}$ of a non-trivial constant function $c$.
Let $\bar y\subset \gimel \cdot v$. 
Since $\Gamma$ is connected as an oriented graph, there exists a sequence 
$$
\bar x_0=\bar x,\bar x_1,\ldots,\bar x_{n-1},\bar x_n=\bar y
$$ 
such that the path $\bar x_0 \to \bar x_1\to \cdots \to \bar x_n $ connects $\bar x$ 
and $\bar y$. Assume, by induction, that we proved that $1_{\bar x_{s-1}}\in N$. 
From our assumptions, there is  $A=\sum f_i\phi_{\xi_i}\in \mathcal A$  
that sends $1_{\bar x_{s-1}}$ to $a_{\bar x_{s-1}}$ with a constant non-trivial 
representative $a\ne 0$, see (\ref{eq A(F)_eta}). This implies the first
part of claim~\eqref{thmnn54.2} and the second part of claim~\eqref{thmnn54.2}
follows from the first part of claim~\eqref{thmnn54.2} and claim~\eqref{thmnn54.1}.
\end{proof}

Let $M(\bar\xi)$ denote the $\mathcal A$-submodule of 
$M(G_{\gimel\cdot v}, \gimel \cdot v )$ generated by $\mathbb E^{G_{\gimel\cdot v}}_{\bar\xi}$.
The simple quotient $M(\bar\xi)/N(\bar\xi)$, whose existence is 
guaranteed by Theorem~\ref{thmnn54}\eqref{thmnn54.1}, is 
the {\em canonical } Harish-Chandra  $\mathcal A$-module  associated to $\bar\xi$.

\noindent
V.~M.: Department of Mathematics, Uppsala University, Box. 480,
SE-75106, Uppsala, SWEDEN, email: {\tt mazor\symbol{64}math.uu.se}

\noindent
E.~V.: Departamento de Matem{\'a}tica, Instituto de Ci{\^e}ncias Exatas,
Universidade Federal de Minas Gerais,
Av. Ant{\^o}nio Carlos, 6627, CEP: 31270-901, Belo Horizonte,
Minas Gerais, BRAZIL, and Laboratory of Theoretical and Mathematical Physics, Tomsk State University, 
Tomsk 634050, RUSSIA, 

\noindent email: {\tt VishnyakovaE\symbol{64}googlemail.com}

\end{document}